\documentclass[9pt, a4paper]{amsart}

\usepackage{amsfonts}
\usepackage{amssymb}
\usepackage{amsmath, amsthm,color}
\usepackage{hyperref}
\usepackage{pdfsync}
\usepackage{comment}
\usepackage{calrsfs}
\DeclareMathAlphabet{\pazocal}{OMS}{zplm}{m}{n}

\DeclareMathAlphabet{\mathcal}{OMS}{cmsy}{m}{n}
\SetMathAlphabet{\mathcal}{bold}{OMS}{cmsy}{b}{n}
\newcommand{\R}{\mathbb{R}}
\newcommand{\C}{\mathbb{C}}

\newtheorem{atheorem}{Theorem}[section]

\newtheorem{theorem}{Theorem}

\newtheorem{proposition}{Proposition}
\newtheorem{remark}{Remark}

\numberwithin{equation}{section}

\begin{document}
	\title{Multiversion of the Hausdorff--Young inequality}
	\author[P.~Ivanisvili]{Paata Ivanisvili}
	\author[P.~Kalantzopoulos]{Pavlos Kalantzopoulos}
	
	\address{Department of Mathematics, University of California Irvine, Irvine, CA
	}
	\email{pivanisv@uci.edu \textrm{(P.\ Ivanisvili)}}
	
	\address{Department of Mathematics, University of California Irvine, Irvine, CA
	}
	\email{pkalantz@uci.edu  \textrm{(P.\ Kalantzopoulos)}}
	\makeatletter
	\@namedef{subjclassname@2010}{
		\textup{2010} Mathematics Subject Classification}
	\makeatother
\begin{abstract}
We consider a family of \emph{jointly Gaussian} random vectors \(\xi_j \in \mathbb{R}^{k_j}\), each standard normal but possibly correlated, and investigate when
\[
\mathbb{E}\, F\!\Bigl(B\bigl(|T_{z_1} f_1(\xi_1)|,\dots,|T_{z_n} f_n(\xi_n)|\bigr)\Bigr)
\;\;\le\;\;
F\!\Bigl(\,\mathbb{E}\,B\bigl(|f_1(\xi_1)|,\dots,|f_n(\xi_n)|\bigr)\Bigr)
\]
holds, where \(T_{z}\) is either a Mehler transform (\(z \in \mathbb{C}\)) or a noise operator (\(z \in \mathbb{R}\)).
This framework unifies and extends real and complex hypercontractivity to multi-function settings, yielding multiversions of the sharp Hausdorff--Young inequality, the log-Sobolev inequality, and a ``noisy'' Gaussian--Jensen inequality.
Applications include a new covariance-based characterization of the Brascamp--Lieb inequality in the presence of noise.
\end{abstract}

	\maketitle

\section{Introduction}

A classical result in harmonic analysis is the \emph{Hausdorff--Young inequality with sharp constant}. It states that for functions $g$ in  $L^{p}(\mathbb{R}^{n})$  and exponents $1<p\leq 2$, $q=\frac{p}{p-1}$, one has
\[
\bigl\|\widehat{g}\bigr\|_{L^q(\mathbb{R}^n)}
\;\le\;
C_{p,q}
\;\bigl\|g\bigr\|_{L^p(\mathbb{R}^n)},
\]
where $\widehat{g}(x) =\frac{1}{(2\pi)^{\frac{n}{2}}} \,\int_{\mathbb{R}^n} g(y)\, e^{-\,i \, x\cdot y}\,\mathrm{d}y$ is the Fourier transform, and the constant \(C_{p,q}\) is best possible. This celebrated theorem was first established by Beckner~\cite{Beckner}. Beckner's original argument began with an inequality on two points (a discrete \(\{-1,1\}\) hypercube setting), then used an induction and a central limit theorem to pass to the Gaussian setting, a result that has come to be called \emph{complex hypercontractivity}. Afterward, a suitable change of variables yielded the Hausdorff--Young inequality in its full sharp form.

Over the years, several other remarkable proofs of this sharp inequality have appeared, together with stability results \cite{Christ}. Janson~\cite{Jans} used a Brownian-motion (stochastic-calculus) approach; Lieb~\cite{Lieb} exploited in a very illuminating way the symmetry of the Gaussian extremizers (``Gaussian kernels have only Gaussian maximizers''); and Hu~\cite{Hu01} essentially replaced the stochastic part of Janson's proof by writing everything in terms of heat flows. Later in \cite{IvanVol} it was noted that the continuous heat-flow argument is, in fact, a limiting version of Beckner's original discrete inductive scheme. This (continuous) flow is sometimes called the \emph{Beckner--Janson flow} \cite{IvanVol}, although it is perhaps less familiar to many experts. It turns out that if one tries to use \emph{standard} heat-flow methods in the usual manner, the monotonic interpolation property for the Hausdorff--Young inequality \emph{fails}, as shown by Bennett, Bez, and Carbery~\cite{BannettBeazCarbery}. In contrast, the subtle Beckner--Janson flow does give a monotonic approach and is the key new ingredient in these proofs.

\medskip

\noindent
\textbf{Multifunction extensions.}
In this paper, we carry forward the Beckner--Janson flow ideas to a \emph{multifunction} setting. In doing so, we revisit several classical sharp estimates in harmonic analysis (such as the aforementioned Hausdorff--Young inequality, the Gaussian hypercontractive inequalities, etc.) and obtain their precise sharp analogues for multiple functions. To get a feel for one of our main results on the sharp multiversion of the Hausdorff--Young inequality, we refer the reader to Theorem~\ref{rhoHSI}, i.e., the “baby” two-function case with equal exponents; in that setting, taking \(f = g\) and \(\rho = 0\)  recovers the sharp Beckner--Babenko inequality.

Beyond these specific corollaries, our main technical contribution lies in an  \emph{FB~theorem}, i.e., Theorem~\ref{complexCaseTheorem}, which provides a fairly general framework for proving sharp multifunction estimates via the Beckner--Janson flow. We show that the classical Brascamp--Lieb inequalities  and our new multifunction Hausdorff--Young inequalities can all be viewed as special cases of this single abstract FB~theorem. While we mostly focus here on how this theorem extends sharp known inequalities to their natural ``multi'' analogues, we believe it may well have broader applicability for other families of inequalities, by allowing one to choose different pairs of \((F,B)\).

\section{Main results}
\noindent
\textbf{Notations.} In this paper $\xi := (\xi_{1}, \ldots, \xi_{n})$ denotes a normal random vector in $\mathbb{R}^{k_{1}+\ldots+k_{n}}$, where each $\xi_{j}$ is a standard normal Gaussian vector in $\mathbb{R}^{k_{j}}$ for $j = 1, \ldots, n$. We remark that $\xi_{1}, \ldots, \xi_{n},$ are not necessarily independent of each other.  Throughout the paper, unless stated otherwise, the exponents $p_{1}, \ldots, p_{n}$ are assumed to be positive, the matrix $\mathrm{cov}(\xi)$ has full rank, $\lambda_{\min}$ and $\lambda_{\max}$ will denote the smallest and the largest eigenvalues of the covariance matrix $\mathrm{cov}(\xi)$. The symbol $\mathcal{P}(\mathbb{R}^{n})$ denotes the space of all polynomials $h : \mathbb{R}^{n} \mapsto \mathbb{C}$, and for any function $f$ we let $\mathcal{P}(\mathbb{R}^{n})f := \{ hf,  \;   h \in \mathcal{P}(\mathbb{R}^{n})\}$. 

\vskip1cm

Our first main result is the extension of complex hypercontractivity to several functions. 
\begin{theorem}[$n$-function complex hypercontractivity]\label{MCH}
Let $\alpha\geq 1$ and $(z_{1}, \ldots, z_{n}) \in \mathbb{C}^{n}$. The inequality 
 \begin{align}\label{multipulti}
 \| \prod_{j=1}^{n} |T_{z_{j}}f_{j}(\xi_{j})|^{p_{j}}\|_{\alpha} \leq \|  \prod_{j=1}^{n} |f_{j}(\xi_{j})|^{p_{j}}\|_{1}
 \end{align}
 holds for all polynomials $f_{j} : \mathbb{R}^{k_{j}} \mapsto \mathbb{C}$ if and only if
\begin{align}\label{localcc}
\sum_{i,j}(\Re w_{i})\mathrm{cov}(\xi_{i}, \xi_{j}) (\Re w_{j})^{T} - \alpha (\Re z_{i}w_{i}) \mathrm{cov}(\xi_{i}, \xi_{j}) (\Re z_{j}w_{j})^{T} \geq \Re \sum_{j} \frac{(1-z_{j}^{2}) w_{j} w_{j}^{T}}{p_{j}}
\end{align}
holds for all $w_{j} \in \mathbb{C}^{k_{j}}$ where $j=1, \ldots, n$.
\end{theorem}
In the theorem $||X||_{p} = (\mathbb{E} |X|^{p})^{1/p}$ denotes the $p$'th moment of a random variable $X$. The Mahler transform $T_{z_{j}}$ acts on a polynomial $f_{j}$ written in Fourier-Hermite basis $f_{j} =\sum_{|\beta|\leq \mathrm{deg}(f_j)}c_{\beta} H_{\beta}$ as follows
$$
T_{z_{j}}f_{j} = \sum_{|\beta|\leq \mathrm{deg}(f_{j})}z_{j}^{|\beta|}c_{\beta} H_{\beta},
$$
where $\beta = (\beta_{1}, \ldots, \beta_{k_{j}}) \in (\mathbb{Z}_{+})^{k_{j}}$ is the multiindex, $|\beta| = \beta_{1}+\ldots+\beta_{k_{j}}$, $c_{\beta}\in \C$ and $H_{\beta}(x) = \mathbb{E}(x+i \xi_{j})^{\beta}$ denotes probabilists' Hermite polynomial, where for a complex vector $w_{j}=(w_{j1}, \ldots, w_{jk_{j}}) \in \mathbb{C}^{k_{j}}$ we set $w_{j}^{\beta} = w_{j1}^{\beta_{1}}\cdots w_{jk_{j}}^{\beta_{j}}$, $z_{j} w_{j} = (z_{j} w_{j1}, \ldots, z_{j} w_{jk_{j}})$, and $\Re w_{j} = (\Re w_{j1}, \ldots, \Re w_{jk_{j}})$ for all $j=1, \ldots, n$. 

\subsection{Case of imaginary $z_{j}$'s} 
If $n=1$ then the condition (\ref{localcc}) simplifies to 
$$
|\Re w_{1}|^{2}-\alpha | \Re z_{1} w_{1}|^{2} \geq \Re \frac{(1-z_{1}^{2})w_{1}^{2}}{p_{1}}
$$
which after setting $p_{1}=p$ and choosing $\alpha=q/p$ for any $q\geq p$ becomes the same condition  as in the classical complex hypercontractivity in Gauss space due to Beckner \cite{Beckner} for $z=i\sqrt{p-1}$, and $q=p/(p-1)$;  Weissler \cite{Wei} (for all exponents $p,q$,  $1<p\leq q<\infty$ except $2<p\leq q \leq 3$ and its dual),  Janson \cite{Jans}, Lieb \cite{Lieb},  and Epperson \cite{Epp}. The complex hypercontractivity has found numerous applications in heat-smoothing conjecture \cite{IvaniEsk1,MenNaor}, Markov--Bernstein inequalities and its reverse forms \cite{IE}, moment comparison inequalities \cite{IvTk, IE}, and it provides $L_{p}$  bounds in a systematic way for Fourier multipliers acting on functions having prescribed frequencies \cite{IvNa, IE}. Perhaps the first application of complex hypercontractivity goes back to Beckner.

\begin{atheorem}[Beckner~\cite{Beckner}]\label{bekth}
Let $X \sim N(0, \mathrm{Id_{n}})$. The inequality
\begin{align}\label{bekeka}
\|T_{i \sqrt{p-1}}f(X)\|_{q}\leq \| f(X)\|_{p},
\end{align}
holds for all polynomials $f \in \mathcal{P}(\mathbb{R}^{n})$ if 
\begin{align*}
    \frac{1}{p}+\frac{1}{q}=1, \quad 1\leq p\leq 2. 
\end{align*}
\end{atheorem}
The inequality \eqref{bekeka} is equivalent to Hausdorff--Young inequality with a sharp constant, namely 
\begin{align}\label{HS1}
\left[\int_{\mathbb{R}^{n}}\left|\widehat{g}\left(x\right)\right|^{q} \mathrm{d}x\right]^{1/q}\leq C\left[\int_{\mathbb{R}^{n}}\left|g\left(x \right)\right|^{p} \mathrm{d}x\right]^{1/p}
\end{align}
holds for all $g= \mathcal{P}(\mathbb{R}^{n})e^{-|x|^{2}/2p}$ with the best possible constant 
\begin{align*}
C=\frac{p^{n/2p}}{q^{n/2q}} (2\pi)^{\frac{n}{2q}-\frac{n}{2p}}, \quad \text{and} \quad  \widehat{g}(x):= \frac{1}{(2\pi)^{n/2}} \int_{\mathbb{R}^{n}} g(y) e^{-i x \cdot y}dy. 
\end{align*}

The connection of $T_{z}f$ for imaginary $z$ and Fourier transform $\widehat{f}$ can be seen through Mehler's formula

\begin{align*}
    T_{z}f(x) = \int_{\mathbb{R}^{n}}f(y) \frac{1}{(\sqrt{1-z^{2}})^{n}}\exp\left(-\frac{|x|^{2}+|y|^{2}}{2} \left(\frac{z^{2}}{1-z^{2}}\right)+x\cdot y \frac{z}{1-z^{2}}\right) \mathrm{d}\gamma(y),
\end{align*}
where $d\gamma(y) = \frac{e^{-|y|^{2}/2}}{(2\pi)^{n/2}}dy$ is the standard Gaussian measure on $\mathbb{R}^{n}$. Thus, for $t\geq 1$
		\begin{align}\label{forfur1}
			T_{i\sqrt{t-1}}f(x) &=\frac{e^{\frac{|x|^{2}(t-1)}{2t}}}{ t^{n/2}}\widehat{fe^{-\frac{|y|^{2}}{2t}}}\left(-\frac{x\sqrt{t-1}}{t}\right).
		\end{align}
        Using the well-known scaling property of the Fourier transform, one can verify that the function $e_{t}(x) := e^{-|x|^{2}/2t}$ has Fourier transform $  \widehat{e_{t}}(\omega) = t^{n/2}e^{-|\omega|^{2}t/2}$. Thus formula \eqref{forfur1} rewrites
		\begin{align}\label{forfur}
			T_{i\sqrt{t-1}}f(x)= \frac{\widehat{fe_{t}}}{\widehat{e_{t}}}\left( -\frac{x\sqrt{t-1}}{t}\right).
	\end{align}
Therefore, Beckner's inequality (\ref{bekeka}) can be seen as another way of rewriting Hausdorff--Young inequality (\ref{HS1}) in a form where the sharp constant becomes to be equal to $1$, and Gaussians can be easily seen to be extremizers 
\begin{align}\label{HSd1}
	\left\|(\widehat{g}/ \widehat{e_{p}})(\eta)\right\|_{q} \leq \left\| (g/ e_{p})(X)\right\|_{p} \quad \text{where} \quad \eta= \frac{\sqrt{p-1}}{p}X, \quad X \sim N\left(0, \mathrm{Id}_{n}\right).
\end{align}

To summarize, the connection between (\ref{HSd1}) and Hausdorff--Young inequality (\ref{HS1}) is that $g=fe_{p}$. When $z_{j}$'s are purely imaginary numbers in Theorem~\ref{MCH} then the condition (\ref{localcc}) simplifies and we obtain 

\begin{theorem}[$n$-function complex hypercontractivity for imaginary parameters]\label{MBEK}
    Let $\alpha\geq 1$ and $(s_{1}, \ldots, s_{n}) \in \mathbb{R}^{n}$. The inequality 
    \begin{align}\label{nfuComImagi}
        \| \prod_{j=1}^{n} |T_{is_{j}}f_{j}(\xi_{j})|^{p_{j}}\|_{\alpha}\leq \| \prod_{j=1}^{n} |f_{j}(\xi_{j})|^{p_{j}}\|_{1}
    \end{align}
    holds for all polynomials $f_{j} \in \mathcal{P}(\mathbb{R}^{k_{j}})$, $j=1, \ldots, n,$ if and only if 
\begin{equation}\label{MBEKloc}
\frac{1}{\alpha}\mathrm{diag}\left\{ \frac{1+s_{j}^{2}}{s_{j}^{2}p_{j}}\mathrm{Id}_{k_{j}} \right\}_{j=1}^{n} \geq \mathrm{cov}(\xi) \geq \mathrm{diag}\left\{ \frac{1+s_{j}^{2}}{p_{j}}\mathrm{Id}_{k_{j}} \right\}_{j=1}^{n}.
\end{equation}
\end{theorem}

In the theorem the symbol $\mathrm{diag}$ stands for a diagonal matrix, and the diagonals are filled by block-diagonal matrices. The notation $A\geq B$ for matrices $A, B$ means that the matrix $A-B$ is positive semidefinite.

Theorem~\ref{MBEK} may seem a bit abstract compared to Beckner's Thoerem~\ref{bekth}. If we let $p_{1}=\ldots=p_{n}=p$, $\alpha =q/p$, and $s_{1}=s_{2}=\ldots=s_{n}$ in Theorem~\ref{MBEK} then we obtain a result which extends Beckner's result to several functions.

\begin{theorem}[Multiversion of Beckner's inequality]\label{MBEKS} 
   The inequality 
    \begin{align}\label{MBEKSglob}
    \left\| \prod_{j=1}^{n} T_{i\sqrt{p \lambda_{\min}-1}}f_{j}(\xi_{j})\right\|_{q} \leq \left\| \prod_{j=1}^{n} f_{j}(\xi_{j})\right\|_{p}
    \end{align}
    holds for all polynomials $f_{j} \in \mathcal{P}(\mathbb{R}^{k_{j}})$, $j=1, \ldots, n,$ if 
    \begin{align}\label{MBEKSloc}
        \frac{1}{p \lambda_{\min}} + \frac{1}{q \lambda_{\max}}=1, \quad 1\leq p\leq q. 
    \end{align}
\end{theorem}
In the case $n=1$ we have $\lambda_{\min}=\lambda_{\max}=1$ and, therefore, we recover  Beckner's inequality (\ref{bekeka}). The case $n=2$ of Theorem~\ref{MBEKS} gives the following 
\begin{theorem}[$\rho$-correlated Beckner's inequality]\label{rcor}
Let $X$ and  $X_{\rho}$ be $\rho$-correlated standard Gaussians in $\mathbb{R}^{n}$, i.e., $\mathrm{cov}(X, X_{\rho})=\rho \mathrm{Id}_{n}$,  having joint Gaussian distribution. The inequality 
\begin{align*}
   \| T_{i\sqrt{p(1-|\rho|)-1}}f(X)T_{i\sqrt{p(1-|\rho|)-1}}g(X_{\rho})\|_{q} \leq \|f(X)g(X_{\rho})\|_{p}
\end{align*}
holds for all polynomials $f, g \in  \mathcal{P}(\mathbb{R}^{n})$  if 
\begin{align*}
    \frac{1}{p(1-|\rho|)}+\frac{1}{q(1+|\rho|)}=1, \quad 1\leq p\leq q. 
\end{align*}
\end{theorem}
\begin{remark}
    Notice that $0$-correlated case of Theorem~\ref{rcor} with $f=g$ coincides with Theorem~B due to Beckner. 
\end{remark}

Theorem~\ref{MBEK} together with the identity (\ref{forfur}) will give us a general form  of Hausdorff--Young inequality for several functions. The following two corollaries present reformulations of Theorem~\ref{MBEK} and Theorem~\ref{MBEKS} in the language of Fourier transforms. We chose to state them separately to make the results accessible and engaging to readers from both the harmonic analysis and probability audiences, allowing each to appreciate the statements in their own familiar framework.

\begin{theorem}[general $n$-function Hausdorff--Young inequality]\label{nhsin}
 Let $\alpha\geq 1$ and $\eta_{j} = \frac{\sqrt{t_{j}-1}}{t_{j}} \xi_{j} $, $j=1, \ldots, n$.   Then 	
		\begin{align}\label{hutRATIOineq}
			\left\|\prod_{j=1}^{n} \left|(\widehat{g}_{j}/ \widehat{e_{t_{j}}})(\eta_{j})\right|^{p_{j}}\right\|_{\alpha} \leq \left\|\prod_{j=1}^{n} \left|(g_{j}/ e_{t_{j}})(\xi_{j})\right|^{p_{j}}\right\|_{1}
		\end{align}
		holds for all $g_{j}\in \mathcal{P}(\mathbb{R}^{k_{j}}) e_{t_{j}}$, $j=1, \ldots, n,$ if and only if
        \begin{align}\label{MBEKloc1}
\frac{1}{\alpha}\mathrm{diag}\left\{ \frac{t_{j}}{(t_{j}-1)p_{j}}\mathrm{Id}_{k_{j}} \right\}_{j=1}^{n} \geq \mathrm{cov}(\xi) \geq \mathrm{diag}\left\{ \frac{t_{j}}{p_{j}}\mathrm{Id}_{k_{j}} \right\}_{j=1}^{n}.
        \end{align}
\end{theorem}
The next Theorem is a reformulation of Theorem~\ref{MBEKS} in terms of Fourier transform. 
\begin{theorem}[$L^{p}$ to $L^{q}$ version of  $n$-function Hausdorff--Young inequality]\label{pqHS}
Let  $\mu := \frac{\sqrt{p\lambda_{\min}-1}}{p\lambda_{\min}}$.    The inequality 
   		\begin{align}\label{nHS}
			\left\|e^{\frac{|\xi|^{2}}{2q \lambda_{\max}}}\prod_{j=1}^{n} \widehat{g}_{j}(\mu \xi_{j})\right\|_{q} \leq (p\lambda_{\min})^{\frac{\sum k_{j}}{2}} \left\|e^{\frac{|\xi|^{2}}{2p \lambda_{\min}}} \prod_{j=1}^{n} g_{j}(\xi_{j})\right\|_{p}
		\end{align}
   holds for all $g_{j}\in \mathcal{P}(\mathbb{R}^{k_{j}}) e_{p \lambda_{\min}}$, $j=1, \ldots, n,$ if
    \begin{align*}
        \frac{1}{p \lambda_{\min}} + \frac{1}{q \lambda_{\max}}=1, \quad 1\leq p\leq q. 
    \end{align*}
    The constant $(p\lambda_{\min})^{\frac{\sum k_{j}}{2}}$ is sharp in (\ref{nHS}) and it is achieved if   $g_{j}(\xi_{j})=e_{p \lambda_{\min}}(\xi_{j})$ for all $j=1, \ldots, n$. 
    
\end{theorem}

The application of the Theorem \ref{pqHS} to $n=2$ gives the following.
\begin{theorem}[$\rho$-correlated Hausdorff--Young inequality]\label{rhoHSI}
Let $\rho \in [0,1)$, and the exponents $q\geq p \geq 1$ satisfy 
\begin{align}\label{rhoHSIloc}
    \frac{1}{p(1-\rho)}+\frac{1}{q(1+\rho)}=1.
\end{align}
Then the inequality
\begin{align*}    \left[\int_{\mathbb{R}^{n}}\int_{\mathbb{R}^{n}}|\widehat{f}(x)\widehat{g}(y) |^{q} e^{-\frac{\rho p q|x-y|^{2}}{2}}\mathrm{d}x\mathrm{d}y  \right]^{1/q} \leq C\left[\int_{\mathbb{R}^{n}}\int_{\mathbb{R}^{n}}\left|f(x)g(y) \right|^{p}e^{\frac{\rho|x+y|^{2}}{2(1-\rho^{2})}}\mathrm{d}x\mathrm{d}y \right]^{1/p}
\end{align*}
holds for all  $f,g \in \mathcal{P}(\mathbb{R}^{n})e_{p(1-\rho)}$, and the best possible constant 
$$
C = \frac{(p(1-\rho))^{n-\frac{n}{q}}}{(q(1+\rho))^{\frac{n}{q}}} \frac{1}{\left(2\pi \sqrt{1-\rho^{2}}\right)^{\frac{n}{p}-\frac{n}{q}}}
$$
which is achieved for $f=g=e_{p(1-\rho)}$.  
\end{theorem}

Notice that $0$-correlated case of the theorem with $f=g$ coincides with Beckner--Babenko inequality (\ref{HS1}) with the sharp constant. 

\vskip1cm
Another application of Theorem~\ref{MBEK} is reverse H\"older inequality for $d$-homogeneous Gaussian chaoses. We say that the polynomial $f : \mathbb{R}^{k}\mapsto \mathbb{C}$ is $d$-homogeneous Gaussian chaos if 
$$
f(x) = \sum_{|\alpha|=d}c_{\alpha}H_{\alpha}(x).
$$

\begin{theorem}[reverse H\"older for Gaussain chaoses]\label{chaoses}
 For any $q\geq p \geq 1/\lambda_{\min}$ we have 
  \begin{align*}
      \left\|\prod_{j=1}^{n}f_{j}(\xi_{j})\right\|_{q}\leq \max\left\{ \frac{1}{p\lambda_{\min}-1}, q\lambda_{\max}-1 \right\}^{\sum d_{j}/2}\left\|\prod_{j=1}^{n}f_{j}(\xi_{j})\right\|_{p}
  \end{align*}
  holds for all $d_{j}$-homogeneous Gaussian chaoses $f_{j} :\mathbb{R}^{k_{j}}\mapsto \mathbb{C}$, $j=1, \ldots, n$.
\end{theorem}
In Theorem~\ref{chaoses2}, we will derive similar results using real hypercontractivity for several functions. The advantage of complex hypercontractivity in the reverse Hölder inequality for Gaussian chaoses was only recently noted in \cite{IvTk}, where it was demonstrated that Beckner's inequality (\ref{bekeka}) implies the estimate $\|f(X)\|_{2} \leq e^{d/2} \|f(X)\|_{1}$ for any $d$-homogeneous Gaussian chaos $f$. Currently, $e^{d/2}$ represents the best known upper bound. In contrast, real hypercontractivity yields the upper bound $e^{d}$, which is less sharp. However, for certain values of $p$ and $q$, real hypercontractivity can provide better upper bounds on $\|f(X)\|_{q}/\|f(X)\|_{p}$. For further details, see \cite{IvTk}.

\subsection{Case of real $z_{j}$'s.}
For $r \in [-1,1]$ the noise operator $T_{r}f(x) = \mathbb{E}f(x r + \xi \sqrt{1-r^{2}})$ for all measurable $f :\mathbb{R}^{n} \to [0, \infty)$,  where $\xi \sim N(0, \mathrm{Id}_{n})$.

The next theorem shows that in the case  $(z_{1}, \ldots, z_{n}) \in \mathbb{R}^{n}$ the inequality (\ref{multipulti}) extends to all nonnegative measurable test functions $f_{j}$, and, moroever, the reverse form of Theorem~\ref{MCH} also holds true. 
\begin{theorem}[$n$-function forward and reverse  real hypercontractivity]\label{MRH}
Let $(r_{1}, \ldots, r_{n})\in [-1,1]^{n}$, and let $p_{j} \in \mathbb{R} \setminus \{0\}$, $j=1, \ldots, n$.  

\textup{(i)}
For $\alpha \in (-\infty, 0)\cup [1, \infty)$ the inequality 
 \begin{align}\label{multipultir}
 \| \prod_{j=1}^{n} (T_{r_{j}}f_{j}(\xi_{j}))^{p_{j}}\|_{\alpha} \leq \|  \prod_{j=1}^{n} f^{p_{j}}_{j}(\xi_{j})\|_{1}
 \end{align}
 holds for all measurable $f_{j} : \mathbb{R}^{k_{j}} \to (0, \infty)$ if and only if
 \begin{align}\label{localreal}
     \mathrm{diag}\left\{ \frac{1-r_{j}^{2}}{p_{j}  } \mathrm{Id}_{k_{j}}\right\}_{j=1}^{n} \leq \left\{(1-\alpha r_{u}r_{v})\mathrm{cov}(\xi_{u}, \xi_{v})\right\}_{u,v=1}^{n}.
 \end{align}

 \textup{(ii)} For $\alpha \in (0,1]$ the inequality 
 \begin{align}\label{multipultirrev}
 \| \prod_{j=1}^{n} (T_{r_{j}}f_{j}(\xi_{j}))^{p_{j}}\|_{\alpha} \geq \|  \prod_{j=1}^{n} f^{p_{j}}_{j}(\xi_{j})\|_{1}
 \end{align}
  holds for all measurable $f_{j} : \mathbb{R}^{k_{j}} \to (0, \infty)$ if and only if
 \begin{align}\label{localrealrev}
     \mathrm{diag}\left\{ \frac{1-r_{j}^{2}}{p_{j}} \mathrm{Id}_{k_{j}}\right\}_{j=1}^{n} \geq \left\{(1-\alpha r_{u}r_{v})\mathrm{cov}(\xi_{u}, \xi_{v})\right\}_{u,v=1}^{n}.
 \end{align}
\end{theorem}

The condition (\ref{localreal}) is the  same as (\ref{localcc}) when $(z_{1}, \ldots, z_{n})=(r_{1}, \ldots, r_{n})$ are real numbers.

When $n=1$ Theorem~\ref{MRH}   coincides with real hypercontractivity due to Bonami,  Nelson \cite{Bonami, Nelson}, and reverse real hypercontractivity due to Borell \cite{Borell82}.

If we set $r_{1}=r_{2}=\ldots=r_{n}=0$ notice that the left hand side in (\ref{multipultir}) and in (\ref{multipultirrev}) takes the form
$$
\| \prod_{j=1}^{n} (T_{r_{j}}f_{j}(\xi_{j}))^{p_{j}}\|_{\alpha} = \prod_{j=1}^{n}\mathbb{E}(f_{j}(\xi_{j}))^{p_{j}}.
$$
In this case Theorem~\ref{MRH} recovers the following result.

\begin{atheorem}[Paouris, Chen, Dafnis \cite{ChenPaourDafn}]\label{PaourisDafnisChen}  Let $p_{j} \in \mathbb{R}\setminus \{0\}, \, j=1, \ldots, n$. The inequality 
 \begin{equation}\label{revePaour}
 	\prod_{j=1}^n\left(\mathbb{E} f_j(\xi_j)\right)^{p_j} \leq \mathbb{E} \prod_{j=1}^nf_j(\xi_j)^{p_j}
\end{equation}
holds for all measurable functions $f_{j}: \mathbb{R}^{k_{j}} \to (0,\infty)$, $j=1,\ldots,n$ if and only if
\begin{equation}\label{localPaourisDafnisChen}
  \mathrm{diag}\left(\frac{1}{p_1}\mathrm{Id}_{k_1},\ldots,\frac{1}{p_n}\mathrm{Id}_{k_n}\right) \leq \mathrm{cov}(\xi)
 \end{equation}
 Moreover, the inequality \eqref{revePaour} reverses if and only if the inequality \eqref{localPaourisDafnisChen} is reversed.
\end{atheorem}

As observed in~\cite{ChenPaourDafn}, the reverse inequality to~\eqref{revePaour} admits a geometric interpretation as the Brascamp--Lieb inequality \cite{BL76,Lieb}, once the usual geometric condition is replaced by the covariance-dominance condition (\emph{i.e.}, the reverse of~\eqref{localPaourisDafnisChen}). Furthermore, through a limiting argument in~\eqref{revePaour}, one recovers Barthe’s reverse Brascamp--Lieb result \cite{Barthe98}. This perspective unifies forward and reverse inequalities under a single structural framework. Many developments and applications of these ideas appear throughout the literature, including works of Ball~\cite{Ball} and Barthe~\cite{Barthe}, who clarified extremizability in various settings; BCCT~\cite{BCCT}, which systematically analyzed the geometry and finiteness of the Brascamp--Lieb exponents; and Carlen--Lieb--Loss~\cite{CarLiebLoss04}, Barthe--Cordero-Erausquin~\cite{BarCord04}, and Barthe--Huet~\cite{BartheHuent09}, where functional-analytic and entropy-based methods were further developed. Lehec~\cite{Lehec14} produced streamlined probabilistic proofs of the classical and reverse Brascamp--Lieb estimates, while subsequent refinements and generalizations appear in works such as Bennett--Bez--Carbery~\cite{BannettBeazCarbery}, Barthe--Cordero-Erausquin--Ledoux--Maurey~\cite{BarCordLedMau11}, Courtade--Liu~\cite{CourtaLiu21}, and Anantharam--Jog--Nair~\cite{AnanJoNai22}. More recent studies by Barthe--Wolff~\cite{BarWol22} and Bennett--Tao~\cite{BennTao} explore extremizer classifications, functional extensions, and other structural properties of Brascamp--Lieb-type inequalities.


\subsection{Case of equal $r_{j}$'s and $p_{j}$'s}

If we let  $r_{1}=\ldots=r_{n}=r$, $p_{1}=\ldots=p_{n}=p$ and $\alpha=q/p\geq 1$, then  Theorem~\ref{MRH} part (i) implies the following
\begin{theorem}[correlated $(p,q)$ real hypercontractivity]\label{correal}
   Let $q\geq p \geq 1/\lambda_{\min}$, and pick any $r \in [-1,1]$. The inequality 
   \begin{align}\label{correalFun}
       \left\|\prod_{j=1}^{n} T_{r}f_{j}(\xi_{j})\right\|_{q} \leq \left\|\prod_{j=1}^{n} f_{j}(\xi_{j})\right\|_{p}
   \end{align}
   holds for all measurable $f_{j} :\mathbb{R}^{k_{j}}\mapsto [0, \infty)$ if and only if 
   \begin{align}\label{loceq}
       |r|\leq \sqrt{\frac{p\lambda_{\min}-1}{q\lambda_{\min}-1}}. 
   \end{align}
\end{theorem}

\begin{remark}
    Perhaps the primary example to keep in mind is when $n=2$, and $\xi_{1}$ and $\xi_{2}$ are $\rho$-correlated Gaussians in $\mathbb{R}^{k}$, i.e., $k_{1}=k_{2}=k$ and $\mathrm{cov}(\xi_{1}, \xi_{2})=\rho \mathrm{Id}_{k}$. In this case
    $$
    \lambda_{\min}=1-|\rho|,
    $$
    and the condition (\ref{loceq}) takes the form $|r|\leq \sqrt{\frac{p(1-|\rho|)-1}{q(1-|\rho|)-1}}$. 
    \end{remark}


Let $L:=\Delta -x\cdot \nabla$ be the generator of the Ornstein--Uhlenbeck process in $\mathbb{R}^{n}$. Theorem~\ref{correal} is equivalent (by a standard differentiation argument, see Seciton~\ref{proofsFROMreal}) to the following 
\begin{theorem}[correlated log-Sobolev inequality]\label{corlogsob}
 For any $ p \geq 1/\lambda_{\min}$ we have 
\begin{align}\label{nfunlog}
  &\mathbb{E}\,  \prod_{j=1}^{n}f^{p}_{j}(\xi_{j})\log \prod_{j=1}^{n}f^{p}_{j}(\xi_{j}) - \mathbb{E} \prod_{j=1}^{n}f^{p}_{j}(\xi_{j}) \log \mathbb{E} \prod_{j=1}^{n}f^{p}_{j}(\xi_{j}) \leq\\
  &\frac{p^{2}\lambda_{\min}}{2\sqrt{p\lambda_{\min}-1}} \mathbb{E}\prod_{j=1}^{n}f^{p}_{j}(\xi_{j}) \left(\sum_{j=1}^{n} \frac{Lf_{j}(\xi_{j})}{f_{j}(\xi_{j})}\right) \nonumber
\end{align}
for all measurable $f_{j} :\mathbb{R}^{k_{j}} \mapsto (0, \infty)$, $j=1, \ldots, n$.  Moroever, the constant $\frac{p^{2}\lambda_{\min}}{2\sqrt{p\lambda_{\min}-1}}$ is the best possible.
\end{theorem}

We remark that uncorrelated version of Theorem~\ref{corlogsob}, i.e., when $\xi_{1}, \ldots, \xi_{n}$ are independent of each other, is equivalent to the classical log-Sobolev inequality due to Gross~\cite{Gross}. 
Indeed, in this case one needs to apply log-Sobolev inequality to each $f_{j}$ separately, multiply each of these inequalities by $\mathbb{E}\prod_{k \neq j} f_{k}^{p_{k}}(\xi_{k})$ and sum up the resulted estimates.  

\begin{theorem}[correlated moment comparison estimates]\label{chaoses2}
  For any $q\geq p \geq 1/\lambda_{\min}$ we have 
  \begin{align*}
      \left\|\prod_{j=1}^{n}f_{j}(\xi_{j})\right\|_{q}\leq \left(\frac{q\lambda_{\min}-1}{p\lambda_{\min}-1}\right)^{\sum d_{j}/2}\left\|\prod_{j=1}^{n}f_{j}(\xi_{j})\right\|_{p}
  \end{align*}
  holds for all $d_{j}$-homogeneous Gaussian chaoses $f_{j} :\mathbb{R}^{k_{j}}\mapsto \mathbb{C}$, $j=1, \ldots, n$.
\end{theorem}

\vskip1cm

\subsection{Multiversion $FB$ complex hypercontractivity}
Theorem~\ref{MCH} will be a corollary of a more general result proved in this paper. We say $B: [0, \infty)^{n} \to \mathbb{R}$ satisfies the polynomial growth condition if all second derivatives of $B$ are bounded by a polynomial on $(0, \infty)^{n}$. 

\begin{theorem}[$n$-function $FB$ complex hypercontractivity]\label{complexCaseTheorem}
	Let $F :[0,\infty) \mapsto \mathbb{R}$ and $B: [0,\infty)^{n} \mapsto [0,\infty)$ be smooth in the interior of their domains functions, satisfying the polynomial growth condition and continuous up to the boundary.  We also assume that $F'>0$ on $(0,\infty)$ and $B_m>0$ on $(0,\infty)^n$ for all $1\leq m\leq n$. Let $z=(z_{1}, \ldots, z_{n}) \in \mathbb{C}^{n}$ with $|z_{j}|\leq 1$, $j=1, \ldots, n$. Assume 
    \begin{align}\label{convexity2}
    (t,y) \mapsto \frac{F''(t)}{F'(t)}y^{2} \quad \text{is convex on}\quad  (0,\infty)^{2}. 
    \end{align}
    The inequality 
	\begin{align}\label{CNGJ}
		\mathbb{E} F(B(|T_{z_{1}}f_{1}(\xi_{1})|,\ldots, |T_{z_{n}}f_{n}(\xi_{n})|))\leq F\left( \mathbb{E}B(|f_{1}(\xi_{1})|,\ldots, |f_{n}(\xi_{n})|\right))
	\end{align}
	holds for all polynomials $f_{p}:\mathbb{R}^{k_{p}}\mapsto \mathbb{C}$ if and only if 
	\begin{align}\label{localcomplex}
		&-\sum_{p,q} \left(B_{pq}({\bf{c}})+\frac{F''(B(\bf{c}))}{F'(B(\bf{c}))} 
		B_{p}({\bf{c}})B_{q}({\bf{c}})\right) (\Re z_{p} w_{p})\mathrm{cov}(\xi_{p}, \xi_{q})(\Re z_{q} w_{q})^{T} \\
		&+\sum_{p,q}  B_{pq}({\bf{c}}) (\Re w_{p})\mathrm{cov}(\xi_{p}, \xi_{q})(\Re w_{q})^{T}   +\sum_{p}\frac{B_{p}({\bf{c}})}{c_{p}}  (|\Im w_{p}|^2-|\Im z_pw_{p}|^2)\geq0 \nonumber
	\end{align}
	holds for all $c = (c_{1}, \ldots, c_{n}) \in (0, \infty)^{n}$, and all $w_{p} \in \mathbb{C}^{k_{p}}$, $p=1,\ldots, n$. 
\end{theorem}
We list some of the applications of the theorem. 
\begin{itemize}
    \item[1.] For $F(t)=t^{\alpha}$, $B(c)=c_{1}^{p_{1}}\cdots c_{n}^{p_{n}}$ the theorem coincides with $n$-function complex hypercontractivity, i.e.,  Theorem~\ref{MCH}. 
    \item[2.]  When $n=1$ the theorem recovers $(P,Q)$-complex hypercontractivity,  the main result of \cite{IvanKal} which is a complex extension of  Hariya's theorem \cite{Hariya}. 
    \item[3.] When $F(t)=t$, and $z_{1}=\ldots=z_{n}=0$ we obtain one side of Gaussian--Jensen inequality (see \cite{Neeman14, NeemanPaouris16, IvanVol15, LedouxSlepian}) i.e., 
    $$
    \mathbb{E} B(f_{1}(\xi_{1}), \ldots, f_{n}(\xi_{n}))\geq B(\mathbb{E}f_{1}(\xi_{1}), \ldots, \mathbb{E} f_{n}(\xi_{n}))
    $$
    holds if and only if $\left\{\mathrm{cov}(\xi_{p}, \xi_{q})B_{pq}({\bf c})\right\}_{p,q=1}^{n}\geq 0$, provided that an additional conditions $B_{p}\geq 0$ for $p=1, \ldots, n$ are satisfied. In case of $(z_{1}, \ldots, z_{n})\in \mathbb{R}^{n}$ we will state a bit more general theorem in the next section which will not require mild polynomial growth assumptions on $B$ and $F$, together with $B_{p}>0,$ and $f_{j}$ being polynomials, $j=1, \ldots, n$. 
\end{itemize}

The assumption of $F$ and $B$ being smooth,  in general, is not needed in the statement of Theorem~\ref{complexCaseTheorem}.  We will see that for the implication  (\ref{CNGJ}) implies (\ref{localcomplex}) it suffices $F$ and $B$ to be two times continuously differentiable in the interior of their domains. For the implication (\ref{CNGJ}) implies (\ref{localcomplex})  it sufficies to have $F$ and $B$ to be three times continusly differentiable. In all our applications that we are aware of $F$ and $B$ happen to be smooth functions in the interior of their domains. 

\subsection{Multiversion $FB$ forward and reverse real hypercontractivity.} In real case, i.e.,  $(z_{1}, \ldots, z_{n}) \in \mathbb{R}^{n},$ we are able to obtain forward and reverse inequalities in Theorem~\ref{complexCaseTheorem}. Moreover, we will drop the mild requirements on $F$ and  $B$ to satisfy polynomial growth condition, the test funtions $f_{j}$ to be polynomials, and the extra assumptions $B_{m}>0$ for all $m=1, \ldots, n$. 

Let $C=\prod_{j=1}^{n}C_{j}$ be rectangular box in $\mathbb{R}^{n}$, i.e., product of $n$ closed intervals $C_{j}\subset \mathbb{R}$, and let $J\subset \mathbb{R}$ be any closed interval. 

\begin{theorem}\label{FBrealTheorem}
 Let $(r_1,\ldots,r_n)\in[-1,1]^n$, and let $F : J \to \mathbb{R}$ and $ B : C \to J $ be smooth functions in their domains with $F'(t)\neq0$. The inequality 
\vspace{0.2cm}
\begin{align}\label{NGJ}
\mathbb{E}F ( B(T_{r_1}f_1(\xi_1), \ldots, T_{r_n}f_n( \xi_n)))\leq F\left(	\mathbb{E}B(f_1(\xi_1), \ldots, f_n(\xi_n))\right)
\end{align} 
\vspace{0.2cm}
holds for all measurable functions $f_p:\R^{k_p}\to C_{p}$, $p=1,\ldots,n$,  if and only if 
\vspace{0.2cm}
\begin{align}\label{Nlocal}
\left\{F'(B({\bf{c}}))\left((1-r_pr_q)B_{pq}({\bf{c}})-r_pr_q F''(B({\bf{c}}))B_{p}({\bf{c}})B_{q}({\bf{c}}) \right)\mathrm{cov}(\xi_{p}, \xi_{q}) \right\}_{p,q=1}^n\geq 0.
\end{align}
\vspace{0.2cm}
for all ${\bf{c}}\in C$ provided that 
\begin{align}\label{functional}
(t,y) \mapsto \frac{F''(t)}{|F'(t)|}y^{2}
\end{align}
is convex on $J\times \R$. 
\end{theorem}

We list several applications of Theorem~\ref{FBrealTheorem}. 
\subsubsection{$F$ and $B$ are power functions} The choice $F(t)=t^{\alpha}$ and $B({\bf c})=c_{1}^{p_{1}}\cdots c_{n}^{p_{n}}$ in the theorem  gives $n$-function forward and reverse real hypercontractivity, i.e., Theorem~\ref{MRH}.

\subsubsection{Noisy Gaussian--Jensen inequality}

If we choose $F(t)=t$, then Theorem~\ref{FBrealTheorem} gives a noisy extension of Gaussian--Jensen inequality. We decided to state it as a separate theorem due to its many applications in analysis of Gauss space.
\begin{theorem}[Noisy Gaussian--Jensen]\label{NGJT}
Let $(r_{1}, \ldots, r_{n})\in [-1,1]^{n}$, and let $B : C \mapsto \mathbb{R}$ be smooth in the interior of their domain and continuous up to the boundary. The inequality 
\begin{align}\label{NGJid}
    \mathbb{E} B(T_{r_{1}}f_{1}(\xi_{1}), \ldots, T_{r_{n}}f_{n}(\xi_{n})) \leq \mathbb{E} B(f_{1}(\xi_{1}), \ldots, f_{n}(\xi_{n}))
\end{align}
holds for all measurable $f_{j} : \mathbb{R}^{k_{j}} \mapsto  C_{j}$, $j=1, \ldots, n$,  if and only if 
\begin{align}\label{localNGJ}
    \left\{(1-r_{p}r_{q})B_{pq}({\bf c}) \mathrm{cov}(\xi_{p}, \xi_{q}) \right\}_{p,q=1}^{n}\geq 0
\end{align}
for all ${\bf c} \in \mathrm{int}(C)$.   
\end{theorem}
Clearly the reverse inequality to (\ref{NGJid}) holds if and only if the reverse inequality to (\ref{localNGJ}) holds. The reverse direction of the theorem follows just by applying the theorem to $-B$ instead of $B$. 

The application of the theorem to the case $r_{1}=\ldots=r_{n}=0$ gives Gaussian-Jensen inequality.
\begin{atheorem}[Gaussian--Jensen inequality]\label{GJI}
Let $B : C \mapsto \mathbb{R}$ be smooth in the interior of their domain and continuous up to the boundary. The inequality 
\begin{align}\label{GJid}
     B(\mathbb{E}f_{1}(\xi_{1}), \ldots, \mathbb{E}f_{n}(\xi_{n})) \leq \mathbb{E} B(f_{1}(\xi_{1}), \ldots, f_{n}(\xi_{n}))
\end{align}
holds for all measurable $f_{j} : \mathbb{R}^{k_{j}} \mapsto  C_{j}$, $j=1, \ldots, n$,  if and only if 
\begin{align}
    \left\{B_{pq}({\bf c}) \mathrm{cov}(\xi_{p}, \xi_{q}) \right\}_{p,q=1}^{n}\geq 0
\end{align}
for all ${\bf c} \in \mathrm{int}(C)$.   
\end{atheorem}
Gaussian-Jensen inequality for many functions first appeared in the work of Neeman~\cite{Neeman14}. It also appeared independently in \cite{IvanVol2015} for homogeneous functions $B$. See \cite{LedouxSlepian} for a beautiful exposition of this inequality. 

\subsubsection{Noisy Borell's theorem}\label{noiboy}

Let $s \in [-1,1]$ and let $X$ and $X_{s}$ be $s$ correlated standard Gaussians in $\mathbb{R}^{n}$ such that the joint distribution $(X, X_{s})$ is Gaussian in $\mathbb{R}^{2n}$. Set
\begin{align*}
    M(u,v; s) := \mathbb{P}(X\in H(u), \; X_{s} \in H(v))
\end{align*}
for all $u,v \in (0,1)$, where $H(u)$ and $H(v)$ are any two parrallel halfspaces in $\mathbb{R}^{n}$ having Gaussian densities $\gamma(H(u))=u$, and $\gamma (H(v))=v$. For example, we can take $H(u) = (-\infty, \Phi^{-1}(u))\times \mathbb{R}^{n-1}$, where $\Phi(t) = \frac{1}{\sqrt{2\pi}}\int_{-\infty}^{t} e^{-\ell^{2}/2}d\ell$.

It is known that (see \cite{MosselNeeman15}) for $s \in [0,1]$ we have 
\begin{align}\label{matin01}
    \begin{pmatrix}
        M_{uu} & s M_{uv}\\
        s M_{uv} & M_{vv}
    \end{pmatrix} \leq 0,
\end{align}
and the reverse inequality holds if $s \in [-1,0]$. Perhaps one of the important properties of $M$ is that it satisfies modified Monge--Amp\`ere equation, i.e., 
\begin{align*}
    M_{uu}M_{vv}-s^{2} M_{uv}=0
\end{align*}
for all $s \in [-1,1]$ and all $u,v \in (0,1)$, see \cite{MosselNeeman15, IvanVol15} for the details. The following corollary of Theorem~\ref{NGJT} gives a noisy extension of Borell's noise stability theorem. 

\begin{theorem}\label{Borelnoise}[Noisy Borell's theorem]
  For any  $r_{1}, r_{2} \in [-1,1]$,  and any $s \in [-1,1]$ set $\rho := s \frac{\sqrt{(1-r_{1}^{2})(1-r_{2}^{2})}}{1-r_{1}r_{2}}$. For any Borel measurable sets $A, B \subset \mathbb{R}^{n}$ we have 
\begin{align}\label{Borelnoiseq}
  \mathbb{P}(X \in A, X_{\rho} \in B) \leq  \mathbb{E} M\left(\left|\frac{A-r_{1}X}{\sqrt{1-r_{1}^{2}}}\right|_{\gamma}, \left|\frac{B-r_{2}X_{\rho}}{\sqrt{1-r_{2}^{2}}}\right|_{\gamma}; s\right)
\end{align}
holds for all $s \in (0,1)$, and the reverse inequality if $s \in (-1,0)$. Here $|C|_{\gamma}$ denotes the standard gaussian measure of a measurable set $C \subset \mathbb{R}^{n}$. 
\end{theorem}
The inequality is sharp, i.e., the inequality becomes equality when $A$ and $B$ are parallel halfspaces (see the end of Section~\ref{proofsFROMreal}). 

When $r_{1}=r_{2}=0$ the inequality coincides with Borell's noise stability theorem
\begin{atheorem}[Borell's noise stability~\cite{Borell85}]
For any $s \in [0,1]$ and all measurable subsets $A, B\subset \mathbb{R}^{n}$ we have 
\begin{align}\label{BNS}
     \mathbb{P}(X \in A, X_{s} \in B) \leq  M(|A|_{\gamma}, |B|_{\gamma}; s),
\end{align}  
and the reverse inequality holds if $s \in [-1,0]$. 
\end{atheorem}
Borell’s noise stability theorem has found numerous applications across mathematics, including theoretical computer science, where it plays a central role in the Majority is Stablest theorem (see~\cite{DeMosselNeeman16, MOO10}). The inequality~(\ref{BNS}) also implies the sharp Gaussian isoperimetric inequality; see~\cite{LedouxSlepian}.

Multiversion of (\ref{BNS}) was obtained by Neeman \cite{Neeman14} who showed that if the centered Gaussian vector $(\zeta_{1}, \ldots, \zeta_{n}) \in \mathbb{R}^{n}$  has the property $\mathrm{cov}(\zeta_{i}, \zeta_{j})=m_{ij}\geq 0$, and $m_{ii}=1$, $j=1, \ldots, n$. Then 
\begin{align*}
    N(u_{1}, \ldots, u_{n}) := \mathbb{P}(\zeta_{1}\leq \Phi^{-1}(u_{1}), \ldots, \zeta_{n}\leq \Phi^{-1}(u_{n}))
\end{align*}
defined on $(u_{1}, \ldots, u_{n})\in (0,1)^{n}$ satisfies the inequality 
\begin{align}\label{neem1}
    \left\{ N_{pq}m_{pq}\right\}_{p,q=1}^{n}\leq 0.
\end{align}
Moreover, the determinant of the matrix above is zero for all $(u_{1}, \ldots, u_{n})\in (0,1)^{n}$. It follows from (\ref{neem1}) that $\{N_{pq}m_{pq}\}_{p,q}\otimes \mathrm{Id}_{k}\leq 0$. Thus Theorem~\ref{NGJT} at point $r_{1}=\ldots=r_{n}=0$ applied to $B=N$ gives the following corollary due to Neeman~\cite{Neeman14}
\begin{atheorem}[Multiversion of Borrel's noise stability]
Let $(\xi_{1}, \ldots, \xi_{n})\in \mathbb{R}^{nk}$ be Gaussian random vector with $\xi_{j} \sim N(0, \mathrm{Id}_{k})$, $\mathrm{cov}(\xi_{i}, \xi_{j})=m_{ij}\mathrm{Id}_{k}$ for all $i,j=1, \ldots, n$. We have  
\begin{align*}
    \mathbb{P}(\xi_{1} \in A_{1}, \ldots, \xi_{n} \in A_{n})\leq N(|A_{1}|_{\gamma}, \ldots, |A_{n}|_{\gamma})
\end{align*}
for all measuarble sets $A_{j}\in \mathbb{R}^{k}$.
\end{atheorem}

\subsection{An application to Ehrhard inequality}
Notice that by shifting and rescaling test functions if necessary Theorem~\ref{GJI} holds true for any Gaussian random vector $(\xi_{1}, \ldots, \xi_{n}) \in \mathbb{R}^{k_{1}+\ldots+k_{n}}$. 

It was proved in \cite{IvanisviliEhrhard} that the application of Theorem~\ref{GJI} to the case $n=3$, $k_{1}=k_{2}=k_{3}=k$, specifically chosen sequence of functions $B_{N}(c_{1}, c_{2}, c_{3}) = H(c_{1}, c_{2})^{N} c_{3}^{\beta(N)}$ (for some map $\beta : [0, \infty) \mapsto \mathbb{R}$), sequence of Gaussian random vectors $(\xi_{1, N}, \xi_{2, N}, \xi_{3, N}) \in \mathbb{R}^{3k}$, and test functions $(f_{1, N}, f_{2, N}, f_{3, N})$ in the limit $N\to \infty$ recovers the following general Ehrhard type  inequality

\begin{atheorem}\label{genehrhard}
  Let $H(x,y)$ be such that  $H \in C^{3}(C_{1}\times C_{2})$ and $H_{x}, H_{y}$ both never vanish on $C_{1}\times C_{2}$. Assume $a, b >0$. The inequality 
    \begin{align} \label{supeh}
        \int_{\mathbb{R}^{n}} \sup_{ax+by=t} H(f(x), g(y)) \mathrm{d}\gamma(t) \geq H\left(\int_{\mathbb{R}^{n}} f(t) \mathrm{d}\gamma(t), \int_{\mathbb{R}^{n}} g(t)\mathrm{d}\gamma(t) \right)
    \end{align}
    holds for all Borell measurable $f : \mathbb{R}^{n} \mapsto C_{1}$ and $g : \mathbb{R}^{n} \mapsto C_{2}$ if  
    \begin{align}\label{loceh}
        a^{2} \frac{H_{xx}}{H_{x}^{2}} + (1-a^{2}-b^{2})\frac{H_{xy}}{H_{x}H_{y}} + b^{2} \frac{H_{yy}}{H_{y}^{2}}\geq 0\quad \text{and} \quad  |1-a^{2}-b^{2}|\leq 2ab.
    \end{align}
\end{atheorem}

In fact, the main result of \cite{IvanisviliEhrhard} shows that \eqref{supeh} also implies \eqref{loceh}. However, this implication does not follow directly from the Gaussian--Jensen inequality; its verification relies on a different set of techniques.

An application of Theorem~\ref{genehrhard} to a specific choice $H(x,y) = \Phi(a\Phi^{-1}(x)+b\Phi^{-1}(y))$, and certain test functions $f$ and $g$ gives the Ehrhard inequality,  i.e.,  for all measurable sets $A, B \subset \mathbb{R}^{n}$ such that $aA+bB$ is measurable we have 
\begin{align*}
 \Phi^{-1}(|aA+bB|_{\gamma}) \geq a\Phi^{-1}(|A|_{\gamma})+b\Phi^{-1}(|B|_{\gamma}).   
\end{align*}
In this inequality parallel haflspaces achieve  the equality. This is a sharp Gaussian analog of a fameous Brunn--Minkowski inequality which among others implies Gaussian isoperimetric inequality \cite{IvanVol15}. 

Another application of Theorem~\ref{genehrhard} with $H(x,y)=x^{\lambda}y^{1-\lambda}$ recovers Prekopa--Leindler inequality, see \cite{IvanisviliEhrhard} for details. We remark that Pr\'ekopa--Leindler inequality can be also obtained via limit argument of Brascamp and Lieb \cite{BrascampLieb766} from Young's convolution inequality which in turn follows from Theorem~\ref{PaourisDafnisChen} applied to 3 functions, see \cite{ChenPaourDafn}.

\subsection{Organization of the paper}

We begin by proving the multiversion $FB$ forward and reverse real hypercontractivity, stated in Theorem~\ref{FBrealTheorem}, in Section~\ref{reverseProof}. In Section~\ref{proofsFROMreal}, we establish Theorems~\ref{MRH}--\ref{chaoses2}. We then turn to the multiversion $FB$ complex hypercontractivity, proving Theorem~\ref{complexCaseTheorem} in Section~\ref{ComplexProof}. Finally, in Section~\ref{proofsFROMcomplex}, we derive Theorems~\ref{MCH}--\ref{chaoses}.

\section{\textbf{Proof of Theorem \ref{FBrealTheorem}}}\label{reverseProof}

While vectors $v \in \mathbb{R}^{n}$ were previously considered as rows, in this section we treat them as columns.
 
Let $d\gamma^{(s)}$ denote the  standard Gaussian measure on $\R^k$ of variance $s>0$ by
\begin{equation*}
\mathrm{d}\gamma^{(s)}(x)=\frac{1}{(2\pi s)^{\frac{k}{2}}}e^{-|x|^2/2s}\,\mathrm{d}x. 
\end{equation*}	
Let $Q_sB$ be the heat flow of a function $B(x)$,  $B:\R^k \mapsto \mathbb{C}$ evaluated at zero, namely 
\begin{equation}\label{heat flow}
	Q_sB=\int_{\R^k} 
	B\,\mathrm{d}\gamma^{(s)}.
\end{equation}
By the heat equation it holds
\begin{equation}\label{heat'-identity1}
	\frac{\mathrm{d}}{\mathrm{d}s}Q_sB=\frac{1}{2}Q_s(\Delta_x B),
\end{equation}
where $\Delta_x$ denotes the Laplacian operator with respect to the variable $x$. For what follows, we need to indicate the variable in which the flow $Q_s$ is applied. For instance, in \eqref{heat'-identity1}, we will denote it as $Q_s^x$ instead of $Q_s$. Let $A_{j} : \mathbb{R}^{k}\to \mathbb{R}^{k_{j}}$ be matrices such that $A_{j}A^{T}_{j}=\mathrm{Id}_{k_{j}}$ for all $j=1, \ldots, n$.

	\begin{proposition}\label{MFstatement}
		 Let $F : [0,1] \to \mathbb{R}$ and $ B : [0,1]^n \to [0,1] $ be functions that are smooth in their domains. Assume  $F'(t)\neq 0$ on $[0,1],$ and
\begin{align}\label{lemmaConvF1}
(t,y) \mapsto \frac{F''(t)}{|F'(t)|}y^{2}
\end{align}
is concave on $[0,1]\times \mathbb{R}$. Let $r=(r_1,\ldots,r_n)\in[-1,1]^n$. The following are equivalent:
	\vspace{0.2cm}
		\begin{enumerate}
			\item \label{Global} \textbf{Global:} For all measurable functions $f_p:\R^{k_p}\to[0,1]$, $p=1,\ldots,n$  it holds
			\vspace{0.2cm}
			\begin{align}\label{FBreal}
				\int_{\R^k}F \Big(B(T_{r_1}&f_1(A_1 x), \ldots, T_{r_n}f_n( A_n x))\Big) \,\mathrm{d}\gamma(x)\\
				&\nonumber\geq F\left(	\int_{\R^k}B(f_1(A_1 x), \ldots, f_n(A_n x))\,\mathrm{d}\gamma(x)\right).
			\end{align}  
			\vspace{0.2cm}
			\item \label{Local}\textbf{Local:} For any ${\bf{c}}\in[0,1]^n$ it holds
			\vspace{0.2cm}
\begin{align}\label{posDefi}
\left\{\left((1-r_pr_q)F'(B({\bf{c}}))B_{pq}({\bf{c}})-r_pr_qF''(B({\bf{c}}))B_{p}({\bf{c}})B_{q}({\bf{c}}) \right)A_pA_q^T \right\}_{p,q=1}^n\leq0
\end{align}
\vspace{0.2cm}
\item \label{monotonicity} \textbf{Monotonicity:} Let $p=1,\ldots,n$ and  $f_p:\R^{k_p}\to[0,1]$ be measurable functions. For $p=1,\ldots,n$ consider the function given by 
\begin{equation}\label{gforMeasurebla}
g_p(u,x,s)=\int_{\mathbb{R}^{k_p}} f_p(A_p u + r_p A_p x + \sqrt{1 - r_p^2} y) \, \mathrm{d}\gamma^{(1-s)}(y),
\end{equation}
where $x,u\in\R^k$, $s\in[0,1]$. Then,  the function $C(s):[0,1]\to \R$ given by  
\vspace{0.2cm}
\begin{align}\label{CinRealcase}
C(s)=Q_{1-s}^x F\big(Q_s^uB\big(g_1(u,x,s)\big),\ldots,g_n(u,x,s))\big)
\end{align}
\vspace{0.2cm}
is nonincreasing.
\end{enumerate} 
\end{proposition}

\begin{proof}		
\eqref{Global}$\implies$\eqref{Local} Let 
\begin{equation}
(\xi_1,\ldots,\xi_n)\sim (A_1x,\ldots,A_nx)
\end{equation}
where $x$ is the standard normal random vector in $\R^k$. Pick $a=(a_{1}, \ldots, a_{n})\in (0,1)^{n}$ and let $\delta = (\delta_{1}, \ldots, \delta_{n})$.  By Taylor's formula we have 
\begin{align*}
    B(a+\delta) = B(a) + \sum_{j=1}^{n} B_{j}(a) \delta_{j} + \frac{1}{2} \sum_{i,j=1}^{n}B_{ij}(a)\delta_{i}\delta_{j} + O\left( \sum_{j=1}^{n} |\delta_{j}|^{3}\right)
\end{align*}
provided that $\|\delta\| \leq \delta'$, where $\delta'$ is sufficiently small. 

Pick any fixed $\omega_{j} \in \mathbb{R}^{k_{j}}$. Set $f_{j}(\xi_{j}) = a_{j} + \varepsilon \omega_{j} \cdot \xi_{j} 1_{\|\xi_{j}\|<\varepsilon^{-1/2}}$, for $j=1, \ldots, n$. Clearly if $\varepsilon<\varepsilon'(\omega)$ is small enough then $\|\{\varepsilon \omega_{j} \cdot \xi_{j} 1_{\|\xi_{j}\|<\varepsilon^{-1/2}}\}_{j=1}^{n}\|<\delta'$. Thus there exists $\varepsilon'=\varepsilon'(\omega, a)$ such that  
\begin{align*}
    &B(...,a_{j} + \varepsilon \omega_{j} \cdot \xi_{j} 1_{\|\xi_{j}\|<\varepsilon^{-1/2}},...) = \\
    &B(a)+\varepsilon \sum_{j=1}^{n}B_{j}(a)\omega_{j} \cdot \xi_{j} 1_{\|\xi_{j}\|<\varepsilon^{-1/2}}+\\
    &\frac{\varepsilon^{2}}{2}\sum_{i,j=1}^{n} B_{ij}(a) \omega_{i} \cdot \xi_{i} 1_{\|\xi_{i}\|<\varepsilon^{-1/2}} \omega_{j} \cdot \xi_{j} 1_{\|\xi_{j}\|<\varepsilon^{-1/2}}+\\
    &\varepsilon^{3} O\left(\sum_{j=1}^{n} (\omega_{j} \cdot \xi_{j})^{3} 1_{\|\xi_{j}\|<\varepsilon^{-1/2}} \right)=: B(a)+\varepsilon I+\frac{\varepsilon^{2}}{2}II+\varepsilon^{3}III
\end{align*}
holds for all $\varepsilon$, $0<\varepsilon<\varepsilon'$.
We claim that $\mathbb{E} I = O(e^{-c/\varepsilon})$ for some $c>0$. Indeed, notice that $\mathbb{P}(\|\xi_{j}\|>\varepsilon^{-1/2}) < 2e^{-C/\varepsilon}$ for some $C>0$. Thus we have 
\begin{align*}
    |\mathbb{E}\,  \omega_{j} \cdot \xi_{j} 1_{\|\xi_{j}\|<\varepsilon^{-1/2}}|  \leq \| \omega_{j} \cdot \xi_{j}\|_{2} \| 1_{\|\xi_{j}\|<\varepsilon^{-1/2}}\|_{2} \leq C' e^{-C/\varepsilon}
\end{align*}
holds for all $j=1, \ldots, n$, where $C$ and  $C'$ are some positive constants. Next, we claim that 
\begin{align*}
\mathbb{E} II = \mathbb{E} \sum_{i,j=1}^{n}B_{ij}(a)\, \omega_{i} \cdot \xi_{i}\,   \omega_{j} \cdot \xi_{j}  + O(e^{-c/\varepsilon}) = \sum_{i,j=1}^{n}B_{ij}(a)\omega_{i}^{T}A_{i}A_{j}^{T}\omega_{j}+ O(e^{-c/\varepsilon})
\end{align*}
for some $c>0$. Indeed, we have
\begin{align*}
   \mathbb{E}\,   \omega_{i} \cdot \xi_{i} 1_{\|\xi_{i}\|<\varepsilon^{-1/2}} \omega_{j} \cdot \xi_{j} 1_{\|\xi_{j}\|<\varepsilon^{-1/2}} &= \mathbb{E}\,   \omega_{i} \cdot \xi_{i} (1-1_{\|\xi_{i}\|\geq \varepsilon^{-1/2}}) \omega_{j} \cdot \xi_{j} (1-1_{\|\xi_{j}\|\geq \varepsilon^{-1/2}})\\&=
   \mathbb{E} \omega_{i}\cdot \xi_{i} \omega_{j} \cdot \xi_{j}+\mathrm{rem},
\end{align*}
where  ``rem'' consists of the sum of the terms which contain the factors of the form $1_{\|\xi_{i}\|\geq \varepsilon^{-1/2}}$. An application of H\"older inequality for many functions shows that $\mathbb{E}\, \mathrm{rem} = O(e^{-c/\varepsilon})$ due to the fact that $\| \omega_{j} \cdot \xi\|_{p}$ are finite for any $p\geq 1$. 

Finally $\mathbb{E} III <C<\infty$ for some $C >0$ independent of $\varepsilon$. Thus we obtain 

\begin{align*}
   \mathbb{E}\,  B(...,a_{j} + \varepsilon \omega_{j} \cdot \xi_{j} 1_{\|\xi_{j}\|<\varepsilon^{-1/2}},...) = B(a) + \frac{\varepsilon^{2}}{2}\sum_{i,j=1}^{n}B_{ij}(a)\, \omega_{i}^{T}A_{i}A_{j}^{T}\omega_{j} + O(\varepsilon^{3})
\end{align*}
So, we obtain 
\begin{align*}
    &F( \mathbb{E}\,  B(...,a_{j} + \varepsilon \omega_{j} \cdot \xi_{j} 1_{\|\xi_{j}\|<\varepsilon^{-1/2}},...)) = \\
    &F(B(a))+F'(B(a))\frac{\varepsilon^{2}}{2}\sum_{i,j=1}^{n}B_{ij}(a)\, \omega_{i}^{T}A_{i}A_{j}^{T}\omega_{j}+O(\varepsilon^{3}).
\end{align*}

Next, let us investigate the behaviour of $\mathbb{E} F(B(f_{1}(\xi_{1}), \ldots, f_{n}(\xi_{n})))$. Let $G = F\circ B$. Let $\zeta_{j}(\xi_{j}) =\varepsilon \omega_{j} \cdot \xi_{j} 1_{\|\xi_{j}\|<\varepsilon^{-1/2}}$ for all $j=1, \ldots, n$. We have 
\begin{align*}
    \| T_{r_{j}} \zeta_{j} \|_{\infty} \leq \|\zeta_{j}\|_{\infty} \leq \|\omega_{j}\| \varepsilon^{1/2} \to 0
\end{align*}
as $\varepsilon \to 0$. 
Next, we have 
\begin{align*}
    \mathbb{E}G(...,a_{j}+T_{r_{j}}(\zeta_{j}),...)&=G(a)+ \sum_{j=1}^{n}G_{j}(a) \mathbb{E}T_{r_{j}}(\zeta_{j})+\\
    &\quad +\frac{1}{2}\sum_{i,j=1}^{n} G_{ij}(a)  \mathbb{E}T_{r_{i}}(\zeta_{i}) T_{r_{j}}(\zeta_{j})+\\
    &\quad +O\left(\sum_{j=1}^{n} \mathbb{E}|T_{r_{j}}(\zeta_{j})|^{3} \right)\\
    &=: G(a)+I'+II'+III'. 
\end{align*}

Let $\eta_{j} =  \varepsilon \omega_{j}\cdot \xi_{j}1_{\|\xi_{j}\|\geq \varepsilon^{-1/2}}$. We have 
\begin{align*}
    &(T_{r_{j}}\zeta_{j})(\xi_{j})=\varepsilon r_{j} \omega_{j}\cdot \xi_{j}-T_{r_j}(\eta_{j}).
\end{align*}
By contractivity we have that for any $p\geq 1$, $\|T_{r_j}(\eta_{j})\|_{p}\leq \|\eta_{j}\|_{p} \leq C_{1}(p) e^{-C_{2}(p)/\varepsilon}$ for some $C_{1}(p), C_{2}(p)>0$. Therefore, similarly as before, one can see that 
\begin{align*}
    &\mathbb{E} T_{r_{j}} \zeta_{j} = O(\varepsilon^{3}),\\
    &\mathbb{E}T_{r_{i}}(\zeta_{i}) T_{r_{j}}(\zeta_{j})=\varepsilon^{2} r_{i}r_{j}\omega_{i}^{T}A_{i}A_{j}^{T}\omega_{j}+O(\varepsilon^{3}),\\
    &\mathbb{E}|T_{r_{j}}(\zeta_{j})|^{3} = O(\varepsilon^{3}).
\end{align*}
Thus we obtain 
\begin{align*}
&\mathbb{E}G(a_{1}+T_{r_{1}}(\zeta_{1}),\ldots,a_{n}+T_{r_{n}}(\zeta_{n})) = \\
&\quad= G(a)+\frac{\varepsilon^{2}}{2}\sum_{i,j=1}^{n} G_{ij}(a)r_{i}r_{j}\omega_{i}^{T}A_{i}A_{j}^{T}\omega_{j} + O(\varepsilon^{3}).   
\end{align*}
In particular, the real global inequality \eqref{FBreal}, implies that 
\begin{align*}
    &G(a)+\frac{\varepsilon^{2}}{2}\sum_{i,j=1}^{n} G_{ij}(a)r_{i}r_{j}\omega_{i}^{T}A_{i}A_{j}^{T}\omega_{j} + O(\varepsilon^{3}) \\
    &\quad \geq F(B(a))+F'(B(a))\frac{\varepsilon^{2}}{2}\sum_{i,j=1}^{n}B_{ij}(a)\, \omega_{i}^{T}A_{i}A_{j}^{T}\omega_{j}+O(\varepsilon^{3}).
\end{align*}
Recall that $G(a)=F(B(a))$. So canceling the constant terms, dividing the both sides of the inequality by $\varepsilon^2$ and taking $\varepsilon \to 0$ we obtain
\begin{align*}
    \sum_{i,j=1}^{n} G_{ij}(a)r_{i}r_{j}\omega_{i}^{T}A_{i}A_{j}^{T}\omega_{j} \geq F'(B(a))\sum_{i,j=1}^{n}B_{ij}(a)\, \omega_{i}^{T}A_{i}A_{j}^{T}\omega_{j}
\end{align*}
which is the same as the local conditoin \eqref{posDefi} in the Proposition \ref{MFstatement}.

		\vspace{0.2cm}
\eqref{Local}$\implies$\eqref{monotonicity}: 
Without loss of generality we can assume that $f_p$ are smooth functions, otherwise replace $f_p$ by $T_{\delta}f_p$ and then let $\delta$ goes to one in \eqref{CinRealcase}. Recall, the functions $g_p=g_{p}(u,x,s)$, $p=1,\ldots,n$ in \eqref{gforMeasurebla} are given by
\begin{equation}\label{recallg_p}
g_p=g_p(u,x,s)=\int_{\mathbb{R}^{k_p}} f_p(A_p u + r_p A_p x + \sqrt{1 - r_p^2} y) \, d\gamma^{(1-s)}(y).
\end{equation}
We define the following two quantities 
\begin{align*}
	{\bf{c}}&=(g_1,\ldots,g_n)\in[0,1]^n,\\
	{\bf{a}}&=Q_{s}^{u}B({\bf{c}})\in[0,1].
\end{align*}
With this notation, the function given in \eqref{CinRealcase} is written as $C(s)=Q_{1-s}^xF(Q_{s}^{u}B({\bf{c}}))$, and thus by product rule we obtain
\begin{align*}
	C'(s)&=\left(\frac{\mathrm{d}}{\mathrm{d} s}Q_{1-s}^x\right)F(Q_{s}^{u}B({\bf{c}}))\\
	&\quad+Q_{1-s}^xF'(Q_{s}^{u}B({\bf{c}}))\left(\frac{\mathrm{d}}{\mathrm{d} s}Q_s^u\right)B({\bf{c}}) \\
	&\quad+Q_{1-s}^xF'(Q_{s}^{u}B({\bf{c}})) Q_s^u\frac{\mathrm{d}}{\mathrm{d} s} B({\bf{c}}).
\end{align*}
In turn, by the heat equation we have
\begin{equation*}
	C'(s)=\frac{1}{2}Q_{1-s}^x\left(-\Delta_xF({\bf{a}})+F'({\bf{a}})Q_s^u\Delta_uB({\bf{c}}) +2F'({\bf{a}}) Q_s^u\frac{\mathrm{d}}{\mathrm{d} s} B({\bf{c}})\right).
\end{equation*}
Next, define $\mathcal{L}_j$ the quantity given by
\begin{equation}\label{derC}
	C'(s)=\frac{1}{2}Q_{1-s}^x\left(\sum_{j=1}^k\mathcal{L}_j\right),
\end{equation}
where
\begin{align}\label{L_j}
	\mathcal{L}_j=-	\frac{\partial^2}{\partial x_j^2}F({\bf{a}})+  F'({\bf{a}}) Q_s^u \frac{\partial^2}{\partial u_j^2}B({\bf{c}})+\frac{2}{k}F'({\bf{a}})Q_s^u\frac{\mathrm{d}}{\mathrm{d}s}B({\bf{c}}).
\end{align}
The rest of the proof shows $C'(s)\leq 0$ in three steps.
\vspace{0.2cm}

{\textbf{Step 1:}} Compute the derivatives in \eqref{L_j} and express $\mathcal{L}_j$ in terms of partial derivatives of  $g_p$:

\vspace{0.2cm}

For each $j=1,\ldots,n$, we see that
\begin{equation*}
	\frac{\partial}{\partial x_j}F({\bf{a}})=F'({\bf{a}})	\frac{\partial}{\partial x_j}(Q_s^uB({\bf{c}}))=F'({\bf{a}})\left(Q_s^u\sum_{p=1}^nB_{p}({\bf{c}})(g_p)_{x_j}\right),
\end{equation*}
and so
\begin{align*}
	\frac{\partial^2}{\partial x_j^2}F({\bf{a}})&=F''({\bf{a}})\left(Q_s^u\sum_{p=1}^nB_{p}({\bf{c}})(g_p)_{x_j}\right)^2\\
	&\quad+F'({\bf{a}})Q_s^u\left(\sum_{p,q=1}^{n}B_{pq}({\bf{c}})(g_p)_{x_j}(g_q)_{x_j} \right)\\
	&\quad+F'({\bf{a}})Q_s^u\left( \sum_{p=1}^nB_{p}({\bf{c}})(g_p)_{x_jx_j}\right).
\end{align*}
On the other hand, for each $j=1,\ldots,n$ we have
\begin{align*}
	\frac{\partial^2}{\partial u_j^2}B({\bf{c}})&=\frac{\partial}{\partial u_j}\left(\sum_{p=1}^nB_{p}({\bf{c}})(g_p)_{u_j}\right)\\
	&=\sum_{p,q=1}^{n}B_{pq}({\bf{c}})(g_p)_{u_j}(g_q)_{u_j}+\sum_{p=1}^nB_{p}({\bf{c}})(g_p)_{u_ju_j}.
\end{align*}
Next, denote $f_{p,b}$ and $f_{p,bb}$ the first and second order partial derivatives of \(f_p\) with respect to the $b$-variable, $1 \leq b \leq k_p$. Define,
\begin{align*}
g_{p,b}= \int_{\mathbb{R}^{k_p}} f_{p,b}(A_p u + r_p A_p x + \sqrt{1 - r_p^2} y) \, d\gamma^{(1-s)}(y)\\
g_{p,bb}= \int_{\mathbb{R}^{k_p}} f_{p,bb}(A_p u + r_p A_p x + \sqrt{1 - r_p^2} y) \, d\gamma^{(1-s)}(y).
\end{align*}
By the heat equation we have
\begin{align*}
	(g_p)_s	=-\frac{1}{2}Q_{1-s}^y (\Delta_y f_p(E))=-\frac{1}{2}\sum_{b=1}^{k_p}(1-r_p^2)g_{p,bb}
\end{align*}
where $E=A_pu+r_pA_px+\sqrt{1-r_p^2}y$. Therefore,
\begin{align*}
	\frac{d}{ds}B({\bf{c}})=\sum_{p=1}^nB_{p}({\bf{c}})\left(g_p\right)_s=-\frac{1}{2}\sum_{p=1}^nB_{p}({\bf{c}})(1-r_p^2)\sum_{ b=1}^{k_p}g_{p,bb}.
\end{align*}
Hence, the quantity in \eqref{L_j}, rewrites in terms of partial derivatives of $g_p$ as
\begin{align}\label{L_jwithg_p}
	\mathcal{L}_j=&-F''({\bf{a}})\left[Q_s^u\sum_{p=1}^nB_{p}({\bf{c}})(g_p)_{x_j}\right]^2\\   
	&\nonumber\quad+F'({\bf{a}})Q_s^u\left(\sum_{p,q=1}^{n}B_{pq}({\bf{c}})\Big[(g_p)_{u_j}(g_q)_{u_j}-(g_p)_{x_j}(g_q)_{x_j}\Big] \right)\\
	&\nonumber\quad+F'({\bf{a}})Q_s^u\left(\sum_{p=1}^nB_{p}({\bf{c}})\Big[(g_p)_{u_ju_j}-(g_p)_{x_jx_j}-\frac{1-r_p^2}{k}\sum_{ b=1}^{k_p}g_{p,bb}\Big]\right).
\end{align}

\vspace{0.2cm}

{\textbf{Step 2:}} Concavity assumption \eqref{lemmaConvF1} moves $Q^u_s$ from inside the parentheses in \eqref{derC} to outside:

\vspace{0.2cm}

Let us assume $F'(t)>0$ on $[0,1]$. Recall ${\bf{a}}=Q_{s}^{u}B({\bf{c}})$. We define a new quantity $\Psi_j$ by removing from the expression $\mathcal{L}_j/F'({\bf{a}})$ the $Q^u_s$,
\begin{align}\label{Psi_j}
	\Psi_j=&-\frac{F''({B(\bf{c}}))}{F'(B({\bf{c}}))}\left[\sum_{p=1}^nB_{p}({\bf{c}})(g_p)_{x_j}\right]^2\\   
	&\quad\nonumber+\left(\sum_{p,q=1}^{n}B_{pq}({\bf{c}})\Big[(g_p)_{u_j}(g_q)_{u_j}-(g_p)_{x_j}(g_q)_{x_j}\Big] \right)\\
	&\quad\nonumber+\left(\sum_{p=1}^nB_{p}({\bf{c}})\Big[(g_p)_{u_ju_j}-(g_p)_{x_jx_j}-\frac{1-r_p^2}{k}\sum_{b=1}^{k_p}g_{p,bb}\Big]\right).
\end{align}
By assumption, the function $(t,y)\mapsto \frac{F''(t)}{F'(t)}y^2$ is concave on $[0,1]\times \mathbb{R}$, and so by Jensen's inequality we have that, for any bounded measurable $Y:\R^k\to\R$ it holds
\begin{align}\label{Jensen}
\nonumber- \frac{F''({\bf{a}})}{F'({\bf{a}})}(Q_s^uY)^2 &=-\frac{F''(Q_{s}^{u}B({\bf{c}}))}{F'(Q_{s}^{u}B({\bf{c}}))}(Q_s^uY)^2  \\
	&\leq- Q_{s}^{u}\left( \frac{F''(B({\bf{c}}))}{F'(B({\bf{c}}))}Y^2\right).
\end{align}
Hence, dividing \eqref{L_jwithg_p} by $F'({\bf{a}})$ and then using the last inequality for $Y=\sum_{p=1}^nB_{p}({\bf{c}})(g_p)_{x_j}$ we obtain
\begin{equation}\label{afterJensen}
	\frac{\mathcal{L}_j}{F'({\bf{a}})}\leq Q_s^u\Psi_j.
\end{equation}
Therefore, applying the last inequality to \eqref{derC} yields
\begin{equation}\label{CwithPsi}
\frac{C'(s)}{F'({\bf{a}})}\leq\frac{1}{2}Q_{1-s}^xQ_s^u\left(\sum_{j=1}^k\Psi_j\right).
\end{equation}
Observe that when $F'(t)<0$ on $(0,1)$ the inequalities in \eqref{Jensen}, \eqref{afterJensen} are reversed and so in \eqref{CwithPsi} as well. Consequently $C'(s)\leq0$ follows from the following   
\vspace{0.2cm}

{\textbf{Step 3:}} The quantity $\sum_{j=1}^k\Psi_j$ is nonpositive  when $F'(t)>0$ and nonnegtive when $F'(t)<0$.

\vspace{0.2cm}

Set 
\begin{align*}
	\mathcal{Q}_1&:=\sum_{j=1}^k\left[\sum_{p=1}^n B_{p}({\bf{c}})(g_p)_{x_j}\right]^2\\
	\mathcal{Q}_2&:=\sum_{j=1}^k\left[(g_p)_{u_j}(g_q)_{u_j}-(g_p)_{x_j}(g_q)_{x_j}\right]\\
	\mathcal{Q}_3&:= \sum_{j=1}^k\Big[(g_p)_{u_ju_j}-(g_p)_{x_jx_j}-\frac{1-r_p^2}{k}\sum_{b=1}^{k_p}g_{p,bb}\Big].
\end{align*}
Rewriting \eqref{Psi_j} we get
\begin{equation}\label{sumPsi}
	\sum_{j=1}^k\Psi_j=-\frac{F''(B({\bf{c}}))}{F'(B({\bf{c}}))}\mathcal{Q}_1+\sum_{p,q=1}^{n}B_{pq}({\bf{c}})\mathcal{Q}_2+\sum_{p=1}^nB_p({\bf{c}})\mathcal{Q}_3.
\end{equation}
Recall, 
\begin{equation}\label{g_p}
	g_p = \int_{\mathbb{R}^{k_p}} f_p(A_p u + r_p A_p x + \sqrt{1 - r_p^2} y) \, d\gamma^{(1-s)}(y).
\end{equation}
Let $ e_1, \ldots, e_k $ be an orthonormal basis in $ \mathbb{R}^k $. Let $A_{p,b}$ be the $b$'th row of $A_p$, $1\leq b\leq k_p$. The partial derivatives of $g_p$ with respect to $u_j$ and $x_j$, $j=1,\ldots,n$ are
\begin{align*}		(g_p)_{u_j}&=\sum_{b=1}^{k_p}g_{p,b}(A_{p,b}\cdot e_j)\\
	(g_p)_{x_j}&=r_p \sum_{b=1}^{k_p}g_{p,b}(A_{p,b}\cdot e_j)
\end{align*}
and thus, the second order partial derivatives
\begin{align*}
	(g_p)_{u_ju_j}&=\sum_{b_1, b_2=1}^{k_p}g_{p,b_1b_2}(A_{p,b_1}\cdot e_j)(A_{p,b_2}\cdot e_j)\\
	(g_p)_{x_jx_j}&=r_p^2\sum_{b_1, b_2=1}^{k_p}g_{p,b_1b_2}(A_{p,b_1}\cdot e_j)(A_{p,b_2}\cdot e_j).
\end{align*}
From the above we have
\begin{align*}
\mathcal{Q}_1&=\sum_{j=1}^k\sum_{p,q=1}^{n}\left[B_{p}({\bf{c}})(g_p)_{x_j} \right] \left[B_{q}({\bf{c}})(g_q)_{x_j} \right]\\
  &=\sum_{j=1}^k\sum_{p,q=1}^{n}r_p r_q  B_{p}({\bf{c}}) B_{q}({\bf{c}})\sum_{\substack{1\leq b_1\leq k_p \\ 1\leq b_2\leq k_q}} g_{p,b_1} g_{q,b_2}(A_{p,b_1}\cdot e_j) (A_{q,b_2}\cdot e_j)\\
  &=\sum_{p,q=1}^{n} r_p r_q  B_{p}({\bf{c}}) B_{q}({\bf{c}})\sum_{\substack{1\leq b_1\leq k_p \\ 1\leq b_2\leq k_q}}  g_{p,b_1} A_{p,b_1}A_{q,b_2}^T g_{q,b_2}.
\end{align*}
Next, we notice that for $1\leq p,q\leq n$ the expression 
\begin{equation*}
\sum_{\substack{1\leq b_1\leq k_p \\ 1\leq b_2\leq k_q}} g_{p,b_1} A_{p,b_1}A_{q,b_2}^T g_{q,b_2}
\end{equation*}
can be written as a matrix multiplication of the form
\[
\begin{bmatrix} g_{p,1} & \cdots & g_{p,k_p}
\end{bmatrix} 
\begin{bmatrix}
\ldots & A_{p,1}&\ldots  \\
       & \vdots &  \\
	\ldots  &A_{p,k_p} & \ldots 
\end{bmatrix}
\begin{bmatrix}
\vdots&   &  \vdots\\
A_{q,1} & \dots & 	A_{q,k_q} \\
\vdots &  & \vdots 
\end{bmatrix}
\begin{bmatrix} g_{q,1} \\ \vdots \\ g_{q,k_q}
\end{bmatrix}
\]
Therefore, setting ${\bf{x}_p}^{T}=(g_{p,1},\ldots,g_{p,k_p})$ where $p=1,\ldots,n$, we can write 
\begin{align*}
\mathcal{Q}_1=\sum_{p,q=1}^{n} r_p r_q  B_{p}({\bf{c}}) B_{q}({\bf{c}})( {\bf{x}_p}^TA_{p}A_{q}^T {\bf{x}_q}).
\end{align*}
Next, setting ${\bf{x}}^{T}=({\bf{x}_1},
\ldots,{\bf{x}_n})$, the expression above can be written
\begin{align}\label{1strerm}
	\mathcal{Q}_1={\bf{x}}^T\left\{  r_p r_q  B_{p}({\bf{c}}) B_{q}({\bf{c}}) A_{p}A_{q}^T  \right\}_{p,q=1}^n{\bf{x}}.
\end{align}
In the same way, for any $1\leq p,q\leq n$ we have that
\begin{align*}
\mathcal{Q}_2&=\sum_{\substack{1\leq b_1\leq k_p \\ 1\leq b_2\leq k_q}}  A_{p,b_1}A_{q,b_2}^T   \left[g_{p,b_1}  g_{q,b_2}  -  r_p r_q g_{p,b_1} g_{q,b_2}  \right]\\
	&={\bf{x}_p^T}((1-r_pr_q) A_{p}A_{q}^T){\bf{x}_q},
\end{align*}
and hence
\begin{align}\label{2ndterm}
\sum_{p,q=1}^{n}B_{pq}({\bf{c}})\mathcal{Q}_2={\bf{x}}^T\left\{\left( 1-r_pr_q\right)B_{pq}({\bf{c}})A_pA_q^T  \right\}_{p,q=1}^n{\bf{x}}.
\end{align}
Last, using that $A_pA^T_p=\mathrm{Id}$ we obtain
\begin{align}\label{3rdterm}
	\mathcal{Q}_3&= \sum_{j=1}^k\Big[(g_p)_{u_ju_j}-(g_p)_{x_jx_j}-\frac{1-r_p^2}{k}\sum_{b=1}^{k_p}g_{p,bb}\Big]\\
	&\nonumber=(1-r_p^2)\sum_{b_1, b_2=1}^{k_p} A_{p,b_1}A_{p,b_2}^T g_{p,b_1b_2}-(1-r_p^2)\sum_{b=1}^{k_p}g_{p,bb}\\
	&\nonumber=0.
\end{align}
Therefore, by substituting the quantities \eqref{1strerm}, \eqref{2ndterm} and \eqref{3rdterm} to \eqref{sumPsi} we have
\begin{equation*}
\sum_{j=1}^k\Psi_j={\bf{x}}^T\left\{\left((1-r_pr_q)B_{pq}({\bf{c}})-r_pr_q B_{p}({\bf{c}})B_{q}({\bf{c}}) \frac{F''(B({\bf{c}}))}{F'(B({\bf{c}}))}\right)A_pA_q^T  \right\}_{p,q=1}^n{\bf{x}}.
\end{equation*}
Thus, by \eqref{posDefi}, the above quantity is nonpositive when $F'(t) > 0$ and nonnegative when $F'(t) < 0$, completing the proof of \eqref{Local}$\implies$\eqref{monotonicity}.
		
		\vspace{0.2cm}
		
		\eqref{monotonicity}$\implies$\eqref{Global}:  	By a change of variable we have
		\begin{align*}
g_p(u,x,s)&=\int_{\mathbb{R}^{k_p}} f_p(A_p u + r_p A_p x + \sqrt{1 - r_p^2} y) \, d\gamma^{(1-s)}(y)\\
&=\int_{\R^{k_p}}f_p(A_pu+r_pA_px+\sqrt{(1-s)(1-r_p^2)}y)\,d\gamma (y),
		\end{align*}
and hence  
		\begin{align*}
			g_p(0,x,0)&=T_{r_p}f_p(A_pu)\\
			g_p(u,0,1)&=f_p(A_pu).
		\end{align*}
Again, we change the variable and write
		\begin{align*}
			C(s)&=Q_{1-s}^x F\big(Q_s^uB\big(g_1(u,x,s),\ldots,g_n(u,x,s)\big)\big)\\
			&=\int_{\R^k}F\left(\int_{\R^k}B(g_1(\sqrt{s}u,\sqrt{1-s}x,s),\ldots,g_n(\sqrt{s}u,\sqrt{1-s}x,s))\,d\gamma(u)\right)\,d\gamma(x).
		\end{align*}
		So,
		\begin{align*}
			C(0)&=\int_{\R^k}F\left(\int_{\R^k}B(g_1(0,x,0),\ldots,g_n(0,x,0))\,d\gamma(u)\right)\,d\gamma(x)\\
			&=\int_{\R^k}F\circ B(T_{r_1}f_1(A_1x),\ldots,T_{r_n}f_n(A_nx))\,d\gamma(x)
		\end{align*}
		and
		\begin{align*}
			C(1)&=\int_{\R^k}F\left(\int_{\R^k}B(g_1(u,0,1),\ldots,g_n(u,0,1))\,d\gamma(u)\right)\,d\gamma(x)\\
			&=F\left(\int_{\R^k}B(f_1(A_1u),\ldots,f_n(A_nu))\,d\gamma(u)\right).
		\end{align*}
		Since $C$ is nonincreasing, we have $C(0) \geq C(1)$, which completes the proof.

	\end{proof}

Proposition \ref{MFstatement} implies Theorem \ref{FBrealTheorem}. The local condition \eqref{Nlocal} remains invariant under the affine change of variables $B(\cdot,\ldots,\cdot)\mapsto B(\lambda_1\cdot+c_1,\ldots,\lambda_n\cdot+c_n)$ and also by scaling $B(\cdot)\mapsto \lambda B(\cdot)$. Therefore, without loss of generality, in Theorem \ref{FBrealTheorem}, we may assume $J = [0,1]$ and $C = [0,1]^n$. Hence, equivalence between \eqref{Global} and \eqref{monotonicity} in Proposition \ref{MFstatement} is the same as Theorem \ref{FBrealTheorem}.

\section{\textbf{Proof of Theorems~\ref{MRH}--\ref{chaoses2}, and \ref{Borelnoise}}}\label{proofsFROMreal}

\begin{proof}[Proof of Theorem~\ref{MRH}]

The result follows from Theorem~\ref{FBrealTheorem} by setting $F(t)=t^{\alpha}$ and choosing $B(t_1,
\ldots,t_n)=t_1^{p_1}\cdots t_n^{p_n}$. We remark that these functions are not smooth at the boundary points i.e., when $t=0$, or $t_{j}=0$ for some $j$. The simple and standard remedy to fix this issue (which also will be used in the proofs of other theorems in this paper)  is to  first obtain the inequality in Theorem~\ref{MRH} for test functions $f_{j}+\delta$ for any $\delta>0$, and then pass to the limit $\delta \to 0$. We could further restrict ourselves to bounded test functions because the general case can be obtained by replacing $f_{j}$ via $\min \{f_{j}, N\}$ and then letting $N \to \infty$.

For these functions, we now aim to simplify the expression
\begin{align}
\mathcal{L}_{ij}({\bf{t}}):=\left((1-r_ir_j)B_{ij}({\bf{t}})-r_ir_j\frac{F''(B({\bf{t}}))}{F'(B(\bf{t}))}B_{i}({\bf{t}})B_{j}({\bf{t}}) \right)
\end{align}
where ${\bf{t}}\in (0, \infty)^{n}$. One can verify the $i$-th and $ij$-th partial derivatives equal to 
\begin{align*}
B_i({\bf{t}})&=p_i\frac{B({\bf{t}})}{t_i},\\
B_{ij}({\bf{t}})&=(p_ip_j-\delta_{ij}p_i)\frac{B({\bf{t}})}{t_it_j},
\end{align*}
where $\delta_{ij}=1$ if $i=j$ and $\delta_{ij}=0$ otherwise. 
Consequently, we obtain
\begin{align*}
\mathcal{L}_{ij}({\bf{t}})&=\frac{B({\bf{t}})}{t_it_j}\left((1-r_ir_j)(p_ip_j-\delta_{ij}p_i)-r_ir_j(\alpha-1)p_ip_j\right)\\
&=\frac{B({\bf{t}})p_ip_j}{t_it_j}\left(\delta_{ij}(\frac{r_ir_j-1}{p_j})+(1-\alpha r_ir_j)\right).
\end{align*}
By Theorem~\ref{FBrealTheorem},  the inequality 
 \begin{align}\label{multipultir11}
 \| \prod_{j=1}^{n} (T_{r_{j}}f_{j}(\xi_{j}))^{p_{j}}\|_{\alpha} \leq \|  \prod_{j=1}^{n} f^{p_{j}}_{j}(\xi_{j})\|_{1}
 \end{align}
 holds if and only if 
 \begin{equation}\label{Ltleq0}
   \{\mathcal{L}_{ij}({\bf{t}})\mathrm{cov}(\xi_{i}, \xi_{j})\}_{i,j=1}^n\geq 0 
 \end{equation}
 for all ${\bf{t}}\in (0, \infty)^{n}$ provided that 
\begin{align}\label{convPower}
(t,y) \mapsto \frac{F''(t)}{|F'(t)|}y^{2}=\frac{\alpha(\alpha-1)}{|\alpha|}\frac{y^2}{t}
\end{align}
is convex on $(0,\infty)\times \R$. The first condition implies
 \begin{align}\label{localreal1}
       \mathrm{diag}\left\{ \frac{1-r_{j}^{2}}{p_{j}  } \mathrm{Id}_{k_{j}}\right\}_{j=1}^{n} \leq  \left\{(1-\alpha r_{i}r_{j})\mathrm{cov}(\xi_{i}, \xi_{j})\right\}_{i,j=1}^{n}.
 \end{align}
and convexity (\ref{convPower}) requires $\alpha \in (-\infty,0)\cup [1, \infty)$.

Moreover, the inequality \eqref{multipultir11} reverses if and only the inequality \eqref{localreal1} is reversed provided that $\alpha \in (0,1]$ obtaining the second part of Theorem \ref{MRH}.
\end{proof}

\begin{proof}[Proof of Theorem~\ref{correal}]
Let $p_1=\ldots=p_n=p>0$, $r_1=\ldots=r_n=r$ and $\alpha=q/p\geq1$ in Theorem \ref{MRH}.  The condition \eqref{localreal} takes the form
\begin{equation}
\frac{1 - r^2}{p - q r^2} \mathrm{Id}_{k_{1}+\ldots+k_{n}} \leq \mathrm{cov}(\xi). \end{equation}
This holds if and only if $\frac{1-r^{2}}{p-qr^2 } \leq \lambda_{\min}$ which in turn implies \eqref{loceq}. 
\end{proof}

\begin{proof}[Proof of Theorem~\ref{corlogsob}]
For $r>1$, we consider the function
\begin{equation}
\psi(r)=\log \|\prod_{j=1}^nT_{\varphi(r)}f_j(\xi_j)\|_{r}
\end{equation}
Denote $g_{\varphi(r)}(\xi)=\prod_{j=1}^nT_{\varphi(r)}f_j(\xi_j)$. We compute
\begin{align}\label{derPSI}
\nonumber\psi(r)'&=\left( \frac{1}{r}\log\left(\mathbb{E}g_{\varphi(r)}^{r}\right) \right)'\\
&=-\frac{1}{r^2}\log\left(\mathbb{E} g_{\varphi(r)}^{r}\right)+\frac{1}{r}\frac{\left(\mathbb{E}g_{\varphi(r)}^{r}\right)'}{\mathbb{E}g_{\varphi(r)}^{r}}
\end{align}
Recall that
\begin{equation*}
\frac{\partial}{\partial r}T_{\varphi(r)}f(x)=\frac{\varphi'(r)}{\varphi(r)}L(T_{\varphi(r)}f(x)).
\end{equation*}
Hence,
\begin{align*}
\frac{\partial}{\partial r}g_{\varphi(r)}=\frac{\partial}{\partial r}\prod_{j=1}^nT_{\varphi(r)}f_j&=\sum_{j=1}^{n} \left( \frac{\partial}{\partial t} T_{\varphi(r)} f_j \right) \prod_{\substack{k=1 \\ k \neq j}}^{n} T_{\varphi(r)} f_k\\
&=\frac{1}{\varphi(r)}\prod_{j=1}^nT_{\varphi(r)}f_j\left( \sum_{j=1}^n
\frac{L(T_{\varphi(r)}f_j)}{T_{\varphi(r)}f_j}\right):=Mg_{\varphi(r)}
\end{align*}
Therefore the second term in \eqref{derPSI} takes the form
\begin{align*}
\frac{1}{r}\frac{\left(\mathbb{E} g_{\varphi(r)}^{r}\right)'}{\mathbb{E} g_{\varphi(r)}^{r}}&=\frac{1}{r}\frac{\mathbb{E}g_{\varphi(r)}^{r}(r'\log g_{\varphi(r)}+r\frac{Mg_{\varphi(r)}}{g_{\varphi(r)}})}{\mathbb{E} g_{\varphi(r)}^{r}}\\
&=\frac{1}{r}\frac{\mathbb{E}g_{\varphi(r)}^{r}\log g_{\varphi(r)}}{\mathbb{E} g_{\varphi(r)}^{r}}+\frac{\mathbb{E}g_{\varphi(r)}^{r}\frac{Mg_{\varphi(r)}}{g_{\varphi(r)}}}{\mathbb{E} g_{\varphi(r)}^{r}}\\
&=\frac{1}{r^2} \frac{1}{\mathbb{E} g_{\varphi(r)}^{r}} \left( \mathbb{E}g_{\varphi(r)}^{r}\log g_{\varphi(r)}^{r} + r^2\mathbb{E}g_{\varphi(r)}^{r}\frac{Mg_{\varphi(r)}}{g_{\varphi(r)}}\right)
\end{align*}
Thus
\begin{align*}
    \psi'(r)=\frac{1}{r^2} \frac{1}{\mathbb{E} g_{\varphi(r)}^{r}} \left(\mathbb{E}g_{\varphi(r)}^{r}\log g_{\varphi(r)}^{r} +r^2 \mathbb{E}g_{\varphi(r)}^{r-1}Mg_{\varphi(r)}\right)
\end{align*}
For \( q \geq p \geq 1/\lambda_{\min} \) and \( r = \sqrt{p\lambda_{\min}-1} / \sqrt{q\lambda_{\min}-1} \), we apply Theorem \ref{correal} with \( T_{1/\sqrt{p\lambda_{\min}-1}} f_j \) in place of \( f_j \), 
 and we obtain  
\begin{equation*}
	\left\| \prod_{j=1}^{n} T_{1/\sqrt{q\lambda_{\min}-1}} f_j(\xi_{j}) \right\|_{q}  
	\leq \left\| \prod_{j=1}^{n} T_{1/\sqrt{p\lambda_{\min}-1}} f_j(\xi_{j}) \right\|_{p}.
\end{equation*}
Consequently, for $\varphi(r)=1/\sqrt{r\lambda_{\min}-1}$ the function $\psi(r)$ is nonincreasing. Therefore, the inequality $\psi'(r)\leq 0$ implies
\begin{align*}
 \mathbb{E}g_{\varphi(r)}^{r}\log g_{\varphi(r)}^{r}&\leq -r^2 \mathbb{E}g_{\varphi(r)}^{r-1}Mg_{\varphi(r)}\\
   &=-\frac{r^2\varphi'(r)}{\varphi(r)} \mathbb{E}\prod_{j=1}^nT_{\varphi(r)}f_j^r\left( \sum_{j=1}^n
\frac{L(T_{\varphi(r)}f_j)}{T_{\varphi(r)}f_j}\right)
\end{align*} 
Finally, replacing $T_{\varphi(r)}f_j$ by $f_j$ and setting $r=p\geq 1/\lambda_{\min}$ completes the proof.
\end{proof}

\begin{proof}[Proof of Theorem~\ref{chaoses2}]
  Notice that for $d_j$-homogeneous Gaussian chaoses and $r>0$ we have
\begin{equation*}
\left\|\prod_{j=1}^{n}T_rf_{j}(\xi_{j})\right\|_{q}=\left\|\prod_{j=1}^{n}r^{d_j}f_{j}(\xi_{j})\right\|_{q}=r^{\sum d_j}\left\|\prod_{j=1}^{n}f_{j}(\xi_{j})\right\|_{q}.
\end{equation*}
Applying Theorem \ref{correal}, for $r = \sqrt{\frac{p\lambda_{\min}-1}{q\lambda_{\min}-1}}$ establishes the desired inequality.  
\end{proof}

\begin{proof}[Proof of Theorem~\ref{Borelnoise}]
By Theorem~\ref{NGJT} the inequality 
\begin{align*}
  \mathbb{E} B(f_{1}(X), f_{2}(X_{\rho})) \leq  \mathbb{E} B(T_{r_{1}}f_{1}(X), T_{r_{2}}f_2(X_{\rho}))
\end{align*}
holds for all measurale $f_{1}, f_{2} :\mathbb{R}^{n} \to [0,1]$ if and only if 
\begin{align}\label{maine}
    \begin{pmatrix}
        (1-r_{1}^{2})B_{11} & \rho (1-r_{1}r_{2})B_{12} \\
        \rho (1-r_{1}r_{2})B_{12} & (1-r_{2}^{2})B_{22}
    \end{pmatrix} \leq 0.
\end{align}
 Here $X, X_{\rho}$ are standard $\rho$ correlated Gaussians in $\mathbb{R}^{n}$ so that $(X, X_{\rho})$ is  Gaussian random vector in $\mathbb{R}^{2n}$. 

 Without loss of generality we will assume $r_{1},r_{2} \in (-1,1)$. Pick any $s \in (0,1)$,  and set $\rho = s \frac{\sqrt{(1-r_{1}^{2})(1-r_{2}^{2})}}{1-r_{1}r_{2}}$.  Let $B(u,v)=M(u,v;s)$, where $M(u,v;s)$ is defined as in Section~\ref{noiboy}. The left hand side in (\ref{maine}) takes the form 
 \begin{align}\label{matinesimp}
        \begin{pmatrix}
        (1-r_{1}^{2})M_{11} &  s\sqrt{(1-r_{1}^{2})(1-r_{2}^{2})}M_{12} \\
         s\sqrt{(1-r_{1}^{2})(1-r_{2}^{2})}M_{12} & (1-r_{2}^{2})M_{22}
    \end{pmatrix} = \tilde{D} \begin{pmatrix}
        M_{11} &  sM_{12} \\
         sM_{12} & M_{22}
    \end{pmatrix}\tilde{D},
 \end{align}
 where $\tilde{D}$ it $2\times 2$ diagonal matrix with diagonal entries $\sqrt{1-r_{1}^{2}}$ and $\sqrt{1-r_{2}^{2}}$. The matrix in (\ref{matinesimp}) is negative semidefinite by  (\ref{matin01}). Choosing $f_{1} = {1}_{A}$,  $f_{2}=1_{B}$, and noting that $T_{r_{1}}f_{1}(v) =\mathbb{E} 1_{A}(rv+\sqrt{1-r^{2}}X) = \left| \frac{A-r_{1}v}{\sqrt{1-r_{1}^{2}}}\right|$ and $T_{r_{2}}f_{2}(v) = \left| \frac{B-r_{2}v}{\sqrt{1-r_{2}^{2}}}\right|$ we obtain (\ref{Borelnoiseq}). The case $s \in (-1,0)$ is similar. 

 To see the sharpness of (\ref{Borelnoiseq}) without loss of generality we can assume $n=1$. Let $A = (-\infty, \Phi^{-1}(u))$ and $B = (-\infty, \Phi^{-1}(v))$.  Let $X'$ and $X'_{s}$ be $s$ correlated Gaussians independent from $A$ and $X_{\rho}$. We have 
 \begin{align*}
     \left| \frac{A-r_{1}X}{\sqrt{1-r_{1}^{2}}}\right| = \Phi\left( \frac{\Phi^{-1}(u)-r_{1}X}{\sqrt{1-r_{1}^{2}}}\right) \quad \text{and} \quad 
     \left| \frac{B-r_{2}X_{\rho}}{\sqrt{1-r_{2}^{2}}}\right| = \Phi\left( \frac{\Phi^{-1}(v)-r_{2}X_{\rho}}{\sqrt{1-r_{2}^{2}}}\right).
 \end{align*}

 Therefore 
 \begin{align*}
     &\mathbb{E} M\left(\left|\frac{A-r_{1}X}{\sqrt{1-r_{1}^{2}}}\right|_{\gamma}, \left|\frac{B-r_{2}X_{\rho}}{\sqrt{1-r_{2}^{2}}}\right|_{\gamma}; s\right) = \\
     &\mathbb{E} \, 1_{\left(-\infty, \frac{\Phi^{-1}(u)-r_{1}X}{\sqrt{1-r_{1}^{2}}}\right)}(X') \; 1_{\left(-\infty, \frac{\Phi^{-1}(v)-r_{2}X_{\rho}}{\sqrt{1-r_{2}^{2}}} \right)} (X'_{s}) =\\
     &\mathbb{E} \, 1_{\left(-\infty, \Phi^{-1}(u)\right)}(X'\sqrt{1-r_{1}^{2}}+r_{1}X) \; 1_{\left(-\infty, \Phi^{-1}(v) \right)} (X'_{s}\sqrt{1-r_{2}^{2}}+r_{2}X_{\rho}) = M(u,v; \rho),
 \end{align*}
 where the last inequality follows from the fact that the random variables $X'\sqrt{1-r_{1}^{2}}+r_{1}X$ and $X'_{s}\sqrt{1-r_{2}^{2}}+r_{2}X_{\rho}$ are standard Gaussians with 
 \begin{align*}
     &\mathrm{cov}(X'\sqrt{1-r_{1}^{2}}+r_{1}X, X'_{s}\sqrt{1-r_{2}^{2}}+r_{2}X_{\rho})= \\
     &s \sqrt{1-r_{1}^{2}}\sqrt{1-r_{2}^{2}} + \frac{sr_{1}r_{2}\sqrt{1-r_{1}^{2}}\sqrt{1-r_{2}^{2}}}{1-r_{1}r_{2}} = \rho. 
 \end{align*}
 This finishes the proof of Theorem~\ref{Borelnoise}
\end{proof}

\section{\textbf{Proof of Theorem~\ref{complexCaseTheorem}}}\label{ComplexProof}
Let 
\begin{align*}
(\xi_1,\ldots,\xi_n)\sim (A_1x,\ldots,A_nx),
\end{align*}
where $x$ is the standard normal random vector in $\R^k$. Recall that $A_{j} : \mathbb{R}^{k}\to \mathbb{R}^{k_{j}}$ are matrices such that $A_{j}A^{T}_{j}=\mathrm{Id}_{k_{j}}$ for all $j=1, \ldots, n$.

\begin{proposition}\label{equivalComplex}
Let $F :[0,\infty) \mapsto \mathbb{R}$ and $M: [0,\infty)^{n} \mapsto [0,\infty)$ be two functions that are smooth in the interior of their domains, continuous up to the boundary,  and satisfy polynomial growth conditions. We also assume that $F'>0$ on $(0,\infty)$, the map 
\begin{align}\label{lemmaConvF}
(t,y) \mapsto \frac{F''(t)}{F'(t)}y^{2}
\end{align}
is convex on $(0,\infty)\times \mathbb{R}$, and 
 $M_{m}>0$ on $(0, \infty)^{n}$ for all $m=1, \ldots, n$.    Let $z=(z_{1}, \ldots, z_{n}) \in \mathbb{C}^{n}$ with $|z_{j}|\leq 1$, $j=1, \ldots, n$. The following are equivalent.
\vspace{0.2cm}
\begin{enumerate}
\item \label{Globalcomplex} \textbf{Global:} For all polynomials $f_j:\R^{k_j}\to\C$, $j=1, \ldots, n$,  it holds
\vspace{0.2cm}
\begin{align*}
\int_{\R^k}F\circ M(|T_{z_1}&(f_1(A_1 x))|^2, \ldots, |T_{z_n}(f_n( A_n x))|^2) \,d\gamma(x)\\
&\leq F\left(	\int_{\R^k}M(|f_1(A_1 x)|^2, \ldots, |f_n(A_n x)|^2)\,d\gamma(x)\right).
\end{align*}  
\vspace{0.2cm}
\item \label{Localcomplex}\textbf{Local:}  For all ${\bf{c}}=(c_{1},\ldots, c_{n})\in (0,\infty)^n$ it holds
\vspace{0.2cm}
\begin{align*}
		&-\frac{F''(M({\bf{c}}))}{F'(M({\bf{c}}))}\left(\sum_{m, q=1}^{n} M_{m}({\bf{c}})M_{q}({\bf{c}})  (\Re(z_mw_{m}))^T A_{m} A_{q}^T  \Re(z_qw_{q}) \right)\\
		&+\left(\sum_{m, q=1}^{n}  M_{mq}({\bf{c}}) \left( (\Re(w_{m}))^T A_{m} A_{q}^T  \Re(w_{q})  -  (\Re(z_mw_{m}))^T A_{m} A_{q}^T  \Re(z_qw_{q}) \right) \right)\\
		&+\left(\frac{1}{2}\sum_{m=1}^n\frac{M_{m}({\bf{c}})}{c_m}(1-|z_m|^2) 
		|w_{m}|^2\right)\geq0
	\end{align*}
	for all $w_{m}\in\C^{k_m}$, $1\leq m\leq n$
\vspace{0.2cm}
\item \label{monotonicitycomplex} \textbf{Monotonicity:} Let $\ell_p:\R^{k_p}\to\mathbb{C}$, $p=1,\ldots,n$ be any polynomials. Consider the function 
\begin{equation}\label{gforComplex}
g_p(u,x,s)=Q_{1-s}^yQ_s^v\ell_p((A_p u+iv)+z_p(A_p x+iy)),
\end{equation}
where $u,x\in\R^k$, $s\in[0,1]$. Then the function $C:[0,1]\to \R$ given by 
\begin{align}\label{4Hflows}
C(s)=Q_{1-s}^x F\big(Q_s^uM\big(|g_1(u,x,s)|^2,\ldots,|g_n(u,x,s)|^2)\big)
\end{align}
is nondecreasing. 
\vspace{0.2cm}
\end{enumerate} 
\end{proposition}

\begin{proof}
\eqref{Globalcomplex}$\implies$\eqref{Localcomplex}:		Consider  $f_j:\R^{k_j}\to\C$, $j=1,\ldots,n$, 
		\begin{equation*}
			f_j(x_j) = a_j+\varepsilon (\eta_j\cdot x_j),
		\end{equation*}
		where $a_j\in \mathbb{C}$ are nonzero, $\eta_j\in\C^{k_j}$ and  $\varepsilon >0$. Notice that   $T_{z_j}f_j:\R^{k_j}\to\C$ are given by
		\begin{equation*}
			T_{z_j} f_j(x_j)= a_j+\varepsilon z_j(\eta_j\cdot x_j).
		\end{equation*}
		Set $J(x):=F(M(x))$, $x\in(0,\infty)^n$ and also $g_j(x):=T_{z_j} f_j(A_jx)$, $x\in\R^k$. By the growth condition on $F$ and $M$ we have  that
		\begin{equation*}
			\int_{\R^n}J(|g_1|^2,\ldots,|g_n|^2)^2\,d\gamma(x)\leq C,
		\end{equation*}
		and thus by Cauchy--Schwarz inequality we see that
		\begin{align*}
			\int_{| x|> \varepsilon^{-\frac{1}{2}}}J(|g_1|^{2},\ldots,|g_n|^2)\,d\gamma(x)=\mathcal{O}(e^{-\frac{1}{100\varepsilon}}),
		\end{align*}
		and hence 
		\begin{align*}
			\int_{\R^k}J(|g_1|^{2},\ldots,|g_n|^2)\,d\gamma(x)=\int_{| x|\leq \varepsilon^{-\frac{1}{2}}}J(|g_1|^{2},\ldots,|g_n|^2)\,d\gamma(x)+\mathcal{O}(\varepsilon^3).
		\end{align*}
		For $a\in (0,\infty)^{n}$, Taylor's formula implies that
		\begin{equation*}
			J(a+\delta) = J(a) + \sum_{j=1}^n J_j(a) \delta_j + \frac{1}{2} \sum_{ i,j=1}^{n} J_{ij}(a) \delta_i\delta_j + \mathcal{O}(\sum_{j=1}^n|\delta_j|^3),   
		\end{equation*}
		provided that the norm of $\delta=(\delta_1,\ldots,\delta_n)\in\R^n$ is small enough. We write
		\begin{align*}
			|g_j(x)|^2=|a_j|^2+2\varepsilon \Re(\bar{a}_jz_j(\eta_j\cdot A_jx))+\varepsilon^2|z_j(\eta_j\cdot A_j x)|^2
		\end{align*}
		and we see that as  $\varepsilon\to0$, for any $|x|\leq \varepsilon^{-1/2}$ the norm of the vector
		\begin{equation*}
			\{\delta_j(x)\}_{j=1}^n:=\{2\varepsilon \Re(\bar{a}_jz_j(\eta_j\cdot A_jx))+\varepsilon^2|z_j(\eta_j\cdot A_j x)|^2\}_{j=1}^n
		\end{equation*} 
		goes to zero,  and so for small $\varepsilon>0$   we have 
        
		\begin{align*}
			J(|g_1|^2,\ldots,|g_n|^2)&= J(|a_1|^2,\ldots,|a_n|^2)  +\sum_{j=1}^n J_j(|a_1|^2,\ldots,|a_n|^2) \delta_j(x) \\
			&+ \frac{1}{2} \sum_{ i,j=1}^{n} J_{ij}(|a_1|^2,\ldots,|a_n|^2) \delta_i(x)\delta_j(x) + \mathcal{O}(\sum_{j=1}^n|\delta_j(x)|^3),   
		\end{align*}
		provided that $|x|\leq \varepsilon^{-1/2}$. Using the identity 
		\begin{align*}
			\int_{\R^n}(\eta_i\cdot A_ix)( \eta_j\cdot A_jx)\,d\gamma(x)=\eta_i^TA_iA_j^T \eta_j,
		\end{align*}
		and the fact that $\delta_j$ is a polynomial, we obtain that
		\begin{align*}
			\int_{| x|\leq \varepsilon^{-\frac{1}{2}}}\delta_j(x)\,d\gamma(x)&=\int_{\R^n}2\varepsilon \Re(\bar{a}_jz_j(\eta_j\cdot A_jx))+\varepsilon^2|z_j(\eta_j\cdot A_j x)|^2\,d\gamma(x)+ \mathcal{O}(\varepsilon^3)\\
			&=\varepsilon^2 |z_j|^2\|\eta_j\|^2 + \mathcal{O}(\varepsilon^3).
		\end{align*}
		Similarly, for small $\varepsilon>0$ we have that
		\begin{align*}
			\int_{| x|\leq \varepsilon^{-\frac{1}{2}}}\delta_i\delta_j\,d\gamma(x)
			&=  4\varepsilon^2\int_{\R^k}\Re(\bar{a}_iz_i(\eta_i\cdot A_ix)) \Re(\bar{a}_jz_j(\eta_j\cdot A_jx))\,d\gamma(x)+\mathcal{O}(\varepsilon^3)\\
			&= 4\varepsilon^{2} (\Re(\bar{a}_{i}z_{i}\eta_{i}))^{T} A_{i}A_{j}^{T} \Re(\bar{a}_{j}z_{j}\eta_{j}) + \mathcal{O}(\varepsilon^3).
		\end{align*} 
		Denote $c=(|a_1|^2,\ldots,|a_n|^2)$. The previous two calculations imply
		\begin{align*}
			\int_{| x|\leq \varepsilon^{-\frac{1}{2}}}J(|g_1|^2,\ldots,|g_n|^2)\,d\gamma(x)&= J(c)  +\varepsilon^2\sum_{j=1}^n J_j(c) |z_j|^2\|\eta_j\|^2\\
			&+ 2\varepsilon^2\sum_{1\leq i,j\leq n} J_{ij}(c)(\Re(\bar{a}_{i}z_{i}\eta_{i}))^{T} A_{i}A_{j}^{T} \Re(\bar{a}_{j}z_{j}\eta_{j})\\
   &+ \mathcal{O}(\varepsilon^3).
		\end{align*}
		Let $q_j(x)=f_j(A_jx)$, $x\in\R^k$, $j=1,\ldots,n$. Similarly,	since
		\begin{align*}
			|q_j(x_j)|^2&=|a_j|^2+2\varepsilon \Re(\bar{a}_j(\eta_j\cdot A_jx_j))+\varepsilon^2|\eta_j\cdot A_jx_j|^2
		\end{align*}
		we obtain
		\begin{align*}
			\int_{| x|\leq \varepsilon^{-\frac{1}{2}}}M(|q_1|^2,\ldots,|q_n|^2)d\gamma &= M(c)  +\varepsilon^2\sum_{j=1}^n M_j(c) \|\eta_j\|^2\\
			&+ 2\varepsilon^2\sum_{1\leq i,j\leq n} M_{ij}(c)(\Re(\bar{a}_{i}\eta_{i}))^{T} A_iA_j^T\Re(\bar{a}_{j}\eta_{j})\\
   &+ \mathcal{O}(\varepsilon^3).
		\end{align*}
		For $t_{0}=M(c)$ we have 
		\begin{align*}
			F(t_{0}+\varepsilon_{1}) = F(t_{0}) +F'(t_{0})\varepsilon_{1}+\mathcal{O}(\varepsilon_{1}^{2}),
		\end{align*}
		as $\varepsilon_{1} \to 0^{+}$. Thus, we have
		\begin{align*}
			F\left( \int_{\mathbb{R}} M(|f_1|^2,\ldots,|f_n|^2) d\gamma\right) &  = F(M(c))\\
			&\quad+F'(M(|a|^2))\varepsilon^2\sum_{j=1}^n M_j(c) \|\eta_j\|^2\\
			&\quad +F'(M(c)) 2\varepsilon^2\sum_{1\leq i,j\leq n} M_{ij}(c)(\Re(\bar{a}_{i}\eta_{i}))^{T} A_iA_j^T \Re(\bar{a}_{j}\eta_{j})  + \mathcal{O}(\varepsilon^3).
		\end{align*}
		The complex global inequality implies that
		\begin{align*}
			& J(c)  +\varepsilon^2\sum_{j=1}^n J_j(c) |z_j|^2\|\eta_j\|^2\\
			&+ 2\varepsilon^2\sum_{ i,j=1}^{n} J_{ij}(c)(\Re(\bar{a}_{i}z_{i}\eta_{i}))^{T} A_{i}A_{j}^{T} \Re(\bar{a}_{j}z_{j}\eta_{j}) + \mathcal{O}(\varepsilon^3)\\
			&\leq F(M(c))+F'(M(c))\varepsilon^2\sum_{j=1}^n M_j(c) \|\eta_j\|^2\\
			& +F'(M(c)) 2\varepsilon^2 \sum_{ i,j=1}^{n} M_{ij}(c)(\Re(\bar{a}_{i}\eta_{i}))^{T} A_iA_j^T\Re(\bar{a}_{j}\eta_{j})  + \mathcal{O}(\varepsilon^3).
		\end{align*}
		Recall, that $J(c)=F(M(c))$. So canceling the constant terms,  dividing both sides of the inequality by $\varepsilon^{2}$,  and taking $\varepsilon \to 0$ we obtain 
		\begin{align*}
			& \sum_{j=1}^n J_j(c) |z_j|^2\|\eta_j\|^2+ 2 \sum_{ i,j=1}^{n} J_{ij}(c)(\Re(\bar{a}_{i}z_{i}\eta_{i}))^{T} A_{i}A_{j}^{T} \Re(\bar{a}_{j}z_{j}\eta_{j})\leq\\
			&\leq F'(M(c))\sum_{j=1}^n M_j(c) \|\eta_j\|^2+ 2F'(M(c)) \sum_{ i,j=1}^{n} M_{ij}(c)(\Re(\bar{a}_{i}\eta_{i}))^{T} A_iA_j^T \Re(\bar{a}_{j}\eta_{j}).
		\end{align*}
		Denoting $\omega_j = \bar{a}_j\eta_j \in\C^{k_j}$, $j=1,\ldots,n$ the above rewrites
		\begin{align*}
			& \sum_{j=1}^n J_j(c) |z_j|^2\|\eta_j\|^2+ 2 \sum_{ i,j=1}^{n} J_{ij}(c)(\Re(z_{i}\omega_i))^{T} A_{i}A_{j}^{T} \Re(z_{j}\omega_j)\leq\\
			&\leq F'(M(c))\sum_{j=1}^n M_j(c) \|\eta_j\|^2+ 2 F'(M(c))\sum_{ i,j=1}^{n} M_{ij}(c)(\Re(\omega_i))^{T} A_iA_j^T \Re(\omega_j).  
		\end{align*}
		After substituting the values 
		\begin{align*}
			J_i(c) &= F'(M(c)) M_i(c),\\
			J_{ij}(c) &= F''(M(c)) M_i(c)M_j(x) + F'(M(c)) M_{ij}(c),
		\end{align*}
		and then dividing by 2$F'(M(c))$, it becomes 
		\begin{align*}
			&-\frac{F''(M(c))}{F'(M(c))} \sum_{ i,j=1}^{n} M_i(c)M_j(x) (\Re(z_{i}\omega_i))^{T} A_{i}A_{j}^{T} \Re(z_{j}\omega_j)\\
			&+\sum_{ i,j=1}^{n} M_{ij}(c)\left[ (\Re(\omega_i))^{T} A_iA_j^T\Re(\omega_j)-  (\Re(z_{i}\omega_i))^{T} A_{i}A_{j}^{T} \Re(z_{j}\omega_j) \right] \\
			&+ \frac{1}{2}\sum_{j=1}^n M_j(c)(1-|z_j|^2) \|\eta_j\|^2  \geq0
		\end{align*}
		which coincide with the local condition of the theorem.
		
\eqref{Localcomplex}$\implies$\eqref{monotonicitycomplex}: It suffices to show the implication for a translation $\tilde{M}(t_1,\ldots,t_n) = M(t_1+\varepsilon,\ldots,t_n+\varepsilon)$, where $\varepsilon > 0$ is fixed, instead of $M$. The main reason for introducing $\tilde{M}$ is to avoid issues with undefined values such as $M_m(0,\ldots,0)$, $M_{mq}(0,\ldots,0)$, and the fact that $M(0,\ldots,0)$ can be zero  which can make $F'(M(0,\ldots,0))$ and $F''(M(0,\ldots,0))$ undefined. 
		
		First we notice that the local condition for $M$ implies the local condition for $\tilde{M}$. Indeed, notice that $\tilde{M} > M(\varepsilon, \ldots, \varepsilon) > 0$ since $M$ has all partial derivatives positive. Next, it suffices to apply the local inequality at the point $\mathbf{t} + \varepsilon = (t_1 + \varepsilon, \ldots, t_n + \varepsilon)$ instead of $\mathbf{t} = (t_1, \ldots, t_n)$, and use the inequality
		\begin{align*}
			\sum_{m=1}^n\frac{M_{m}({\bf{t}}+\varepsilon)}{t_m+\varepsilon}\frac{1-|z_m|^2}{2}\sum_{b=1}^{k_m}|w_{m,b}|^2\leq 	\sum_{m=1}^n\frac{\tilde{M}_{m}({\bf{t}})}{t_m}\frac{1-|z_m|^2}{2}\sum_{b=1}^{k_m}|w_{m,b}|^2.
		\end{align*}
		Next, we notice that the statement about  monotonicity for $\tilde{M}$ implies the statement about monotonicity for $M$. Indeed, for any $s_{1}, s_{2} \in (0,1)$, $s_{1}<s_{2},$ we have $\tilde{C}(s_{1})\leq \tilde{C}(s_{2})$, where 
		\begin{align*}
			\tilde{C}(s):=Q_{1-s}^x F\big(Q_s^uM\big(\big|Q_{1-s}^yQ_s^v&\ell_1((A_1u+iv)+z_1(A_1x+iy)\big)\big|^2+\varepsilon,\ldots\\
			&\ldots,\big|Q_{1-s}^yQ_s^v\ell_n((A_n u+iv)+z_n(A_n x+iy))\big|^2+\varepsilon\big)\big)\geq0.
		\end{align*}
		Let us denote $t_m=\big|Q_{1-s}^yQ_s^v\ell_n((A_m u+iv)+z_m(A_m x+iy))\big|^2$, $m=1,\ldots,n$. Letting $\varepsilon \to 0$, by continuity, $M(t_1 + \varepsilon, \ldots, t_n + \varepsilon) \to M(t_1, \ldots, t_n)$. Since $M$  satisfies polynomial growth condition, applying the Lebesgue dominated convergence theorem we get also that $Q_s^{u}M(t_1+\varepsilon,\ldots,t_n+\varepsilon)\to Q_s^{u}M(t_1,\ldots,t_n)$. Then, by the continuity of $F$ we apply $F$ to this converges and last for the same reason apply $Q_{1-s}^x$, finally obtaining $C'(s) \geq 0$ for $s \in (0,1)$. 
        
        In what follows, for simplicity, we will continue the proof using the letter $M$ instead of $\tilde{M}$, with the additional assumption on $M$ that all the partial deriavtives are well defined on the boundary of the domain of $M$, and $M(0,0, \ldots, 0)>0$.
		
		For $m=1,\ldots,n$, let $\ell_m:\R^{k_m}\to\R$ be polynomials,  and $g_m$ the quantity given in \eqref{gforComplex}. Set
				\begin{align*}
			&{\bf{G}}=(|g_1|^2,\ldots,|g_n|^2),\\
			&{\bf{A}}=Q_{s}^{u}M({\bf{G}}).
		\end{align*}
		With this notation, the function given in \eqref{4Hflows} is written as $C(s)=Q_{1-s}^xF(Q_{s}^{u}M({\bf{G}}))$, and thus we have 
        \begin{align*}
	C'(s)&=\left(\frac{\mathrm{d}}{\mathrm{d}s}Q_{1-s}^x\right)F(Q_{s}^{u}M({\bf{G}}))\\
	&\quad+Q_{1-s}^xF'(Q_{s}^{u}M({\bf{G}}))\left(\frac{\mathrm{d}}{\mathrm{d} s}Q_s^u\right)M({\bf{G}}) \\
	&\quad+Q_{1-s}^xF'(Q_{s}^{u}M({\bf{G}})) Q_s^u\frac{\mathrm{d}}{\mathrm{d} s} M({\bf{G}}).
\end{align*}

In turn, by the heat equation we have
		\begin{equation*}
			C'(s)=\frac{1}{2}Q_{1-s}^x\left(-\Delta_xF({\bf{A}})+F'({\bf{A}})Q_s^u\Delta_uM({\bf{G}}) +2F'({\bf{A}}) Q_s^u\frac{\mathrm{d}}{\mathrm{d} s} M({\bf{G}})\right).
		\end{equation*}
		Let $\ell_{m,b}$ be the partial derivative of $\ell_m$ with respect to $x_b$ variable. Set
		\begin{equation*}
			g_{m,b}=Q_{1-s}^yQ_s^v    \ell_{m,b}((A_m u+iv)+z_m(A_m x+iy)).
		\end{equation*}
	Similarly we denote $g_{m.bb}$ as the corresponding functions associated with the second order partial derivatives of $\ell_{m,bb}$. Let $A_m^b$ be $b$'th row of the matrix $A_m$, $b=1,\ldots,k_m$,  and let $ e_1, \ldots, e_k $ be the standard orthonormal basis in $ \mathbb{R}^k$.	We see that the partial derivatives $(g_m)_{u_j}$ and $(g_m)_{x_j}$ are scalar multiples of each other, since
		\begin{align*}
			(g_m)_{u_j}&=\sum_{b=1}^{k_m}g_{m,b}(A_{m,b}\cdot e_j),\\
			(g_m)_{x_j}&=z_m \sum_{b=1}^{k_m}g_{m,b}(A_{m,b}\cdot e_j).
		\end{align*}
		Therefore,
	\begin{align*}
		(|g_m|^{2})_{u_j}&= 2\sum_{b=1}^{k_m}(A_{m,b}\cdot e_j) \Re(g_{m,b} \bar{g}_m ),  \\ 
		(|g_m|^{2})_{x_j}&=2\sum_{b=1}^{k_m}(A_{m,b}\cdot e_j) \Re(z_m g_{m,b} \bar{g}_m ).
	\end{align*}
	Moreover,
	\begin{align*}
		(g_m)_{u_ju_j}&=\sum_{b_1, b_2=1}^{k_m}g_{m,b_1b_2}(A_{m,b_1}\cdot e_j)(A_{m,b_2}\cdot e_j)\\
		(g_m)_{x_jx_j}&=z_m^2\sum_{b_1, b_2=1}^{k_m}g_{m,b_1b_2}(A_{m,b_1}\cdot e_j)(A_{m,b_2}\cdot e_j)
	\end{align*}
and thus we have
		\begin{align*}
			(|g_m|^2)_{u_ju_j}&=  \sum_{b_1, b_2=1}^{k_m}(A_{m,b_1}\cdot e_j)(A_{m,b_2}\cdot e_j)g_{m,b_1b_2}\bar{g}_m\\
			&+\sum_{b_1, b_2=1}^{k_m}(A_{m,b_1}\cdot e_j)(A_{m,b_2}\cdot e_j)\bar{g}_{m,b_1b_2}g_m  \\
			&+2|(g_m)_{u_j}|^2  \\
			(|g_m|^{2})_{x_jx_j}&=  \sum_{b_1, b_2=1}^{k_m}z_m^2(A_{m,b_1}\cdot e_j)(A_{m,b_2}\cdot e_j)g_{m,b_1b_2}\bar{g}_m\\
			&+\sum_{b_1, b_2=1}^{k_m}\bar{z}_m^2(A_{m,b_1}\cdot e_j)(A_{m,b_2}\cdot e_j)\bar{g}_{m,b_1b_2}g_m\\
			&  +2|(g_m)_{x_j}|^2.  
		\end{align*}
		Next, the product rule and the heat equation together imply
		\begin{align*}
			(g_m)_s&=\frac{\mathrm{d}}{\mathrm{d}s}Q_{1-s}^yQ_s^v\ell_m(S) \\
			&=\frac{1}{2}Q_{1-s}^yQ_s^v\left(-\sum_{b=1}^{k_m}\ell_{m,bb}(S)(iz_m)^2+\sum_{b=1}^{k_m}\ell_{m,bb}(S)(i)^2\right)\\
			&=\frac{z_m^2-1}{2}\sum_{b=1}^{k_m}g_{m,bb},
		\end{align*}
		where $S=( A_m u+iv)+z_m( A_m x+iy)$, and therefore, 
		\begin{align*}
			\frac{\mathrm{d}}{\mathrm{d}s}M({\bf{G}})&=\sum_{m=1}^nM_{m}({\bf{G}})\left(|g_m|^2\right)_s\\
			&=\sum_{m=1}^nM_{m}({\bf{G}})\left(\bar{g}_m(g_m)_s+g_m(\bar{g}_m)_s\right)\\
			&=\sum_{m=1}^nM_{m}({\bf{G}})\left(\frac{z^{2}_m-1}{2}\sum_{b=1}^{k_m}\bar{g}_mg_{m,bb}+\frac{\bar{z}^{2}_m-1}{2}\sum_{b=1}^{k_m}g_m\bar{g}_{m,bb}\right).
		\end{align*}
		To calculate $\Delta_xF({\bf{A}})$ first we otice that 
		\begin{equation*}
			\frac{\partial}{\partial x_j}F({\bf{A}})=F'({\bf{A}})	\frac{\partial}{\partial x_j}(Q_s^uM({\bf{G}}))=F'({\bf{A}})\left(Q_s^u\sum_{m=1}^nM_{m}({\bf{G}})(|g_m|^2)_{x_j}\right),
		\end{equation*}
		and so
		\begin{align*}
			\frac{\partial^2}{\partial x_j^2}F({\bf{A}})&=F''({\bf{A}})\left(Q_s^u\sum_{m=1}^nM_{m}({\bf{G}})(|g_m|^2)_{x_j}\right)^2\\
			&+F'({\bf{A}})Q_s^u\left(\sum_{m, q=1}^{n}M_{mq}({\bf{G}})(|g_m|^2)_{x_j}(|g_q|^2)_{x_j} \right)\\
			&+F'({\bf{A}})Q_s^u\left( \sum_{m=1}^nM_{m}({\bf{G}})(|g_m|^2)_{x_jx_j}\right).
		\end{align*}
		On the other hand, for $\Delta_uM({\bf{G}})$ we have
		\begin{align*}
			\frac{\partial^2}{\partial u_j^2}M({\bf{G}})&=\frac{\partial}{\partial u_j}\left(\sum_{m=1}^nM_{m}({\bf{G}})(|g_m|^2)_{u_j}\right)\\
			&=\sum_{m, q=1}^{n}M_{mq}({\bf{G}})(|g_m|^2)_{u_j}(|g_q|^2)_{u_j}+\sum_{m=1}^nM_{m}({\bf{G}})(|g_m|^2)_{u_ju_j}.
		\end{align*}
		We can write $C'(s)=\frac{1}{2}Q_{1-s}^x\left(\sum_{j=1}^k\mathcal{L}_j\right)$ where 
		\begin{align*}
			\mathcal{L}_j=-	\frac{\partial^2}{\partial x_j^2}F({\bf{A}})+  F'({\bf{A}}) Q_s^u \frac{\partial^2}{\partial u_j^2}M({\bf{G}})+\frac{2}{k}F'({\bf{A}})Q_s^u\frac{\mathrm{d}}{\mathrm{d}s}M({\bf{G}}).
		\end{align*}
		Substituting we see that
		\begin{align*}
			\mathcal{L}_j=&-F''({\bf{A}})\left[Q_s^u\sum_{m=1}^nM_{m}({\bf{G}})(|g_m|^2)_{x_j}\right]^2\\   
			&+F'({\bf{A}})Q_s^u\left(\sum_{m, q=1}^{n}M_{mq}({\bf{G}})\Big[(|g_m|^2)_{u_j}(|g_q|^2)_{u_j}-(|g_m|^2)_{x_j}(|g_q|^2)_{x_j}\Big] \right)\\
			&+F'({\bf{A}})Q_s^u\left(\sum_{m=1}^nM_{m}({\bf{G}})\Big[(|g_m|^2)_{u_ju_j}-(|g_m|^2)_{x_jx_j}+\frac{z_m^2-1}{k}\sum_{b=1}^{k_m}\bar{g}_mg_{m,bb}+\frac{\bar{z}_m^2-1}{k}\sum_{b=1}^{k_m}g_m\bar{g}_{m,bb}\Big]\right).
		\end{align*}
		Recall $F'>0$. Next, we set
		\begin{align*}
			\Psi_j=&-\frac{F''({M(\bf{G}}))}{F'(M({\bf{G}}))}\left[\sum_{m=1}^nM_{m}({\bf{G}})(|g_m|^2)_{x_j}\right]^2\\   
			&+\sum_{m, q=1}^{n}M_{mq}({\bf{G}})\Big[(|g_m|^2)_{u_j}(|g_q|^2)_{u_j}-(|g_m|^2)_{x_j}(|g_q|^2)_{x_j}\Big] \\
			&+\sum_{m=1}^nM_{m}({\bf{G}})\Big[(|g_m|^2)_{u_ju_j}-(|g_m|^2)_{x_jx_j}+\frac{z_m^2-1}{k}\sum_{b=1}^{k_m}\bar{g}_mg_{m,bb}+\frac{\bar{z}_m^2-1}{k}\sum_{b=1}^{k_m}g_m\bar{g}_{m,bb}\Big].
		\end{align*}
		We notice that $\frac{\mathcal{L}_j}{F'({\bf{A}})}\geq Q_s^u\Psi_j$. Indeed, this follows from Jensen's inequality, namely,  we have
		\begin{align*}
			- \frac{F''({\bf{A}})}{F'({\bf{A}})}(Q_s^uY)^2 &\geq- Q_{s}^{u}\left( \frac{F''(M({\bf{G}}))}{F'(M({\bf{G}}))}Y^2\right)  \qquad \text{for any intgrable} \ Y,
		\end{align*}
		since the function 
		\begin{align*}
			(t,y)\mapsto \frac{F''(t)}{F'(t)}y^2
		\end{align*}
		is convex on $(0,\infty)\times \R$. 

        Thus we have 
        \begin{align*}
            &C'(s)=\frac{1}{2}Q_{1-s}^{x}\left(\sum_{j=1}^{k} \mathcal{L}_{j}\right) \geq \frac{1}{2}Q_{1-s}^{x}\left(F'({\bf A})Q_{s}^{u} \left(\sum_{j=1}^{k} \Psi_{j} \right)\right).
        \end{align*}
        
        We will show that the inequality $\sum_{j=1}^k \Psi_j\geq0$ is the same inequality as the local condition of Proposition \ref{equivalComplex}. To show this, we express the three quantities inside the square brackets in the definition of $\Psi_j$ in terms of $g_{m,b}$. We write
		\begin{align*}
			\mathcal{Q}_1:&=\sum_{j=1}^k\left[\sum_{m=1}^n M_{m}({\bf{G}})(|g_m|^2)_{x_j}\right]^2\\
			&=\sum_{j=1}^k\sum_{m, q=1}^{n}\left[M_{m}({\bf{G}})(|g_m|^2)_{x_j} \right] \left[M_{q}({\bf{G}})(|g_q|^2)_{x_j} \right].
		\end{align*}
	For $m = 1, \ldots, n$, let $w_m \in \mathbb{C}^{k_m}$ be the complex vector with coordinates $w_{m,b} = \bar{g}_m g_{m,b} \in \mathbb{C}$, where $1 \leq b \leq k_m$. Plug in the partial derivatives $(|g_m|^2)_{x_j}$ we get
	\begin{align*}
		\mathcal{Q}_1&=4\sum_{j=1}^k\sum_{m, q=1}^{n}\sum_{\substack{1\leq b_1\leq k_m \\ 1\leq b_2\leq k_q}} (A_{m,b_1}\cdot e_j) (A_{q,b_2}\cdot e_j)M_{m}({\bf{G}}) M_{q}({\bf{G}})\Re(z_m \bar{g}_m g_{m,b_1}) \Re(z_q \bar{g}_q g_{q,b_2})\\
		&=4\sum_{m, q=1}^{n}\sum_{\substack{1\leq b_1\leq k_m \\ 1\leq b_2\leq k_q}} A_{m,{b_1}} A_{q,{b_2}}^T M_{m}({\bf{G}}) M_{q}({\bf{G}})\Re(z_m \bar{g}_m g_{m,b_1}) \Re(z_q \bar{g}_q g_{q,b_2})\\
		&=4\sum_{m, q=1}^{n} M_{m}({\bf{G}})M_{q}({\bf{G}})  (\Re(z_mw_{m}))^T A_{m} A_{q}^T  \Re(z_qw_{q}).
	\end{align*}
		Similarly, we also have that
		\begin{align*}
			\mathcal{Q}_2&:=\sum_{j=1}^k\left[(|g_m|^2)_{u_j}(|g_q|^2)_{u_j}-(|g_m|^2)_{x_j}(|g_q|^2)_{x_j}\right]\\
			&=4 \sum_{\substack{1\leq b_1\leq k_m \\ 1\leq b_2\leq k_q}} A_{m,{b_1}} A_{q,{b_2}}^T \left[\Re(\bar{g}_m g_{m,b_1})  \Re(\bar{g}_q g_{q,b_2})  -  \Re(z_m \bar{g}_m g_{m,b_1})  \Re(z_q \bar{g}_q g_{q,b_2})  \right]\\
			&=4\left( (\Re(w_{m}))^TA_{m} A_{q}^T  \Re(w_{q})  -  (\Re(z_mw_{m}))^T A_{m} A_{q}^T  \Re(z_qw_{q}) \right) .
		\end{align*}
		For the last quantity
		\begin{align*}
		\mathcal{Q}_3:= \sum_{j=1}^k\Bigg[\Bigg.&\sum_{b_1, b_2=1}^{k_m}\left[(A_{m,b_1}\cdot e_j) (A_{m,b_2}\cdot e_j)-z_m^2(A_{m,b_1}\cdot e_j) (A_{m,b_2}\cdot e_j)g_{m,b_1b_2}\bar{g}_m\right]\\
			+&\sum_{b_1, b_2=1}^{k_m}\left[(A_{m,b_1}\cdot e_j) (A_{m,b_2}\cdot e_j)-\bar{z}_m^2(A_{m,b_1}\cdot e_j) (A_{m,b_2}\cdot e_j)\bar{g}_{m,b_1b_2}g_m\right]\\
			+&\frac{z_m^2-1}{k}\sum_{b=1}^{k_m}g_{m,bb}\bar{g}_m +\frac{\bar{z}_m^2-1}{k}\sum_{b=1}^{k_m}\bar{g}_{m,bb}g_m+2(|(g_m)_{u_j}|^2 -|(g_m)_{x_j}|^2)\Bigg.\Bigg],
		\end{align*}
		we see that  moving the sum over $j$ inside the brackets, the first four terms vanish, leaving only the last term, since $A_mA^T_m=Id$. Also,
		\begin{align*}
			\sum_{j=1}^k|(g_m)_{u_j}|^2&=\sum_{j=1}^k\sum_{b_1, b_2=1}^{k_m}\Re((A_{m,b_1}\cdot e_j)g_{m,b_1}\overline{(A_{m,b_2}\cdot e_j)g_{m,b_2}})\\
			&=\sum_{b_1, b_2=1}^{k_m} A_{m,b_1}A_{m,b_2}^T\Re(g_{m,b_1}\bar{g}_{m,b_2})\\
			&=\sum_{b=1}^{k_m}|g_{m,b}|^2,
		\end{align*}
		and since $(g_m)_{x_j}=z_m(g_m)_{u_j}$, the quantity $	\mathcal{Q}_3$ simplifies to 
		\begin{align*}
			\mathcal{Q}_3=2\left(1-|z_m|^2 \right)\sum_{b=1}^{k_m}|g_{m,b}|^2.
		\end{align*}

        Thus we have 
        	\begin{align*}		\sum_{j=1}^k\Psi_j=&-4\frac{F''(M({\bf{G}}))}{F'(M({\bf{G}}))}\left(\sum_{1\leq m,q\leq n}  M_{m}({\bf{G}})M_{q}({\bf{G}}) (\Re(z_mw_{m}))^T A_{m} A_{q}^T  \Re(z_qw_{q})  \right)\\
			&+4\left(\sum_{ m, q=1}^{n} M_{mq}({\bf{G}})\left( (\Re(w_{m}))^T A_{m} A_{q}^T  \Re(w_{q})  -  (\Re(z_mw_{m}))^T A_{m} A_{q}^T  \Re(z_qw_{q}) \right)\right)\\
			&+2\left(\sum_{m=1}^n M_{m}({\bf{G}})\left(1-|z_m|^2 \right)\sum_{b=1}^{k_m}|g_{m,b}|^2 \right).
		\end{align*}
        Therefore the nonnegativity of $\sum_{j=1}^{j}\Psi_{j}$ follows from (\ref{Localcomplex}). Indeed, choose $c_{j} = |g_{j}|^{2}$, $j=1,\ldots, n$,  and $w_{m,b} = \bar{g}_{m}g_{m,b} \in \mathbb{C}$, $b=1, \ldots, k_{m}$, $m=1, \ldots, n$ in (\ref{Localcomplex}). 
	So we obtain $C'(s) \geq 0$.

\eqref{monotonicitycomplex}$\implies$\eqref{Globalcomplex}:	Let $f_j: \mathbb{R}^{k_j} \to \mathbb{C}$, $j = 1, \ldots, n$, be polynomials expressed in the Fourier-Hermite basis, given by  
\begin{equation*}
	f_j(x) = \sum_{|\alpha| \leq \deg(f_j)} c_{\alpha,j} H_\alpha(x),
\end{equation*}
where $H_\alpha(x)$ denotes the Hermite polynomial indexed by the multi-index $\alpha$.  Now, consider the polynomials $\ell_j: \mathbb{R}^{k_j} \to \mathbb{C}$, $j = 1, \ldots, n$, expressed in the monomial basis as  
\begin{equation*}
	\ell_j(x) = \sum_{|\alpha| \leq \deg(f_j)} c_{\alpha,j} x^\alpha,
\end{equation*}
where $x^\alpha = x_1^{\alpha_1} x_2^{\alpha_2} \cdots x_{k_j}^{\alpha_{k_j}}$.
We observe, via a change of variables, that the function $C:[0,1]\to\R$ in \eqref{4Hflows} can be equivalently written as
\begin{equation*}
	C(s)=\int_{\R^k}F\left(\int_{\R^k}M(|q_1(u,x,s)|^2,\ldots,|q_n(u,x,s)|^2)\,d\gamma(u)\right)\,d\gamma(x),
\end{equation*}
where for $j=1,\ldots,n$,
\begin{equation*}
	q_j(u,x,s)=\int_{\R^{k_j}}\int_{\R^{k_j}}\ell_j(\sqrt{s}(A_j u+iv)+z_j\sqrt{1-s}(A_j x+iy))\,d\gamma(v)\,d\gamma(y).
\end{equation*}
Notice that
\begin{align*}
	q_j(u,x,0)&=T_{z_j}f_j(A_j x),\\
	q_j(u,x,1)&=f_j(A_j u).
\end{align*}
Hence, the function $C$ interpolates the inequality at the end points $s=0$ and $s=1$, namely
\begin{align*}
	C(0)&=	\int_{\R^k}F\circ M(|T_{z_1}(f_1(A_1 x))|^2, \ldots, |T_{z_n}(f_n( A_n x))|^2) \,d\gamma(x)\\
	C(1)&=F\left(	\int_{\R^k}M(|f_1(A_1 u)|^2, \ldots, |f_n(A_n u)|^2)\,d\gamma(u)\right),
\end{align*}
and therefore the monotonicity of $C$ finishes the proof.

	\end{proof}
	
\begin{proof}[Proof of Theorem \ref{complexCaseTheorem}]

Let $F$ and $B$ satisfy the local condition of Theorem \ref{complexCaseTheorem}, i.e.,  for some $|z_p|\leq1$ it holds
\begin{align}\label{localforProof}
&-\frac{F''(B(\bf{c}))}{F'(B(\bf{c}))}\sum_{p,q} 
B_{p}({\bf{c}})B_{q}({\bf{c}}) (\Re z_{p} w_{p})^{T}A_{p}A_{q}^{T}(\Re z_{q} w_{q}) \\
&\nonumber+\sum_{p,q}  B_{p,q}({\bf{c}}) \left((\Re w_{p})^{T}A_{p}A_{q}^{T}(\Re w_{q})  - (\Re z_{p} w_{p})^{T}A_{p}A_{q}^{T}(\Re z_{q} w_{q})\right) \\
&\nonumber+\sum_{p}\frac{B_{p}({\bf{c}})}{c_{p}}  (|\Im w_{p}|^2-|\Im z_pw_{p}|^2)\geq0
\end{align}
for all $w_{p} \in \mathbb{C}^{k_{p}}$, $p=1,\ldots, n$. We claim that if we set $B(c_1,\ldots,c_n)=M(c_1^2,\ldots,c_n^2)$  then the local condition (\ref{localforProof}) is the same as the local condition (\ref{Localcomplex}) in Proposition \ref{equivalComplex}. Indeed, let us write $B({\bf{c}})=M({\bf{c}}^2)$. We have 
\begin{align*}
&B_p({\bf{c}})=2c_pM_p({\bf{c}}^2),\\
&B_{pq}({\bf{c}})=4c_pc_qM_{pq}({\bf{c}}^2)+2\delta_{pq}M_p({\bf{c}}^2).
\end{align*}
Substituting this expressions into (\ref{localforProof}) we obtain 
\begin{align*}
&-4\frac{F''(M(\bf{c}^2))}{F'(M(\bf{c}^2))}\sum_{p,q} 
c_pc_qM_{p}({\bf{c}}^2)M_{q}({\bf{c}}^2) (\Re z_{p} w_{p})^{T}A_{p}A_{q}^{T}(\Re z_{q} w_{q}) \\
&\quad+4\sum_{p,q}  c_pc_qM_{pq}({\bf{c}}^2))\left((\Re w_{p})^{T}A_{p}A_{q}^{T}(\Re w_{q})  - (\Re z_{p} w_{p})^{T}A_{p}A_{q}^{T}(\Re z_{q} w_{q})\right) \\
&\quad+2\sum_{p,q}  \delta_{pq}M_p({\bf{c}}^2)\left((\Re w_{p})^{T}A_{p}A_{q}^{T}(\Re w_{q})  - (\Re z_{p} w_{p})^{T}A_{p}A_{q}^{T}(\Re z_{q} w_{q})\right) \\
&\quad+2\sum_{p}M_p({\bf{c}}^2)  (|\Im w_{p}|^2-|\Im z_pw_{p}|^2)\geq0.
\end{align*}
Next, observe that the term above involving $\delta_{pq}$ simplifies to
\begin{align*}
2\sum_{p}  M_p({\bf{c}}^2)\left(|\Re w_{p}|^2-|\Re z_pw_{p}|^2\right)
\end{align*}
which can be combined with the last term. Thus the expression above simplifies to 
\begin{align*}
&-4\frac{F''(M(\bf{c}^2))}{F'(M(\bf{c}^2))}\sum_{p,q} 
c_pc_qM_{p}({\bf{c}}^2)M_{q}({\bf{c}}^2) (\Re z_{p} w_{p})^{T}A_{p}A_{q}^{T}(\Re z_{q} w_{q}) \\
&\quad+4\sum_{p,q}  c_pc_qM_{pq}({\bf{c}}^2))\left((\Re w_{p})^{T}A_{p}A_{q}^{T}(\Re w_{q})  - (\Re z_{p} w_{p})^{T}A_{p}A_{q}^{T}(\Re z_{q} w_{q})\right) \\
&\quad+2\sum_{p}M_p({\bf{c}}^2)  (1-|z_p|^2)|w_{p}|^2\geq0.
\end{align*}
Dividing by four, rescaling $w_p$ to $c_pw_p$ and then substituting $(c_1,\ldots,c_n)$ by $(\sqrt{c_1},\ldots,\sqrt{c_n})$ the local reduces to the local stated in Proposition \ref{equivalComplex} and thus, \eqref{localforProof} is equivalent to: for all polynomials $f_j:\R^{k_j}\to\C$ it holds
\vspace{0.2cm}
\begin{align*}
\int_{\R^k}F\circ M(|T_{z_1}&(f_1(A_1 x))|^2, \ldots, |T_{z_n}(f_n( A_n x))|^2) \,d\gamma(x)\\
&\leq F\left(	\int_{\R^k}M(|f_1(A_1 x)|^2, \ldots, |f_n(A_n x)|^2)\,d\gamma(x)\right).
\end{align*}  
Substituting $B$ back into the expression completes the proof.
\end{proof}

\section{\textbf{Proof of Theorems \ref{MCH}--\ref{chaoses}}}\label{proofsFROMcomplex}

\begin{proof}[Proof of Theorem~\ref{MCH}]
	Let   \( F(t) = t^{\alpha} \) and \( B(t_1, \ldots, t_n) = t_1^{p_1} \cdots t_n^{p_n} \)  in Theorem~\ref{complexCaseTheorem}. In this case the inequality \eqref{multipulti} holds for all polynomials \( f_{j} : \mathbb{R}^{k_{j}} \to \mathbb{C} \) if and only if
	\begin{equation}\label{ABC}
		\mathcal{A}({\bf{t}},w) + \mathcal{B}({\bf{t}},w) + \mathcal{C}({\bf{t}},w) \geq 0
	\end{equation}
	for all \( {\bf{t}} = (t_1,\ldots,t_n) \in (0,\infty)^n \) and all \( w=(w_1,\ldots,w_n)$, $w_i \in \mathbb{C}^{k_i} \), \( i=1,\ldots,n \), where
	\begin{align*}
		\mathcal{A}({\bf{t}},w) &:= -\sum_{i,j} \left(B_{ij}({\bf{t}}) + \frac{F''(B(\bf{t}))}{F'(B(\bf{t}))} 
		B_{i}({\bf{t}}) B_{j}({\bf{t}})\right) (\Re z_{i} w_{j}) \mathrm{cov}(\xi_{i}, \xi_{j}) (\Re z_{i} w_{j})^{T}, \\
		\mathcal{B}({\bf{t}},w) &:= \sum_{i,j} B_{ij}({\bf{t}}) (\Re w_{i}) \mathrm{cov}(\xi_{i}, \xi_{j}) (\Re w_{j})^{T}, \\
		\mathcal{C}({\bf{t}},w) &:= \sum_{j} \frac{B_{j}({\bf{t}})}{t_{j}} \left( |\Im w_{j}|^2 - |\Im z_j w_{j}|^2 \right).
	\end{align*}
What remains is to reduce condition \eqref{ABC} to \eqref{localcc}. We first note that the first and second order partial derivatives of \( B \) are given by  
	\begin{align*}
		B_i({\bf{t}}) &= \frac{B({\bf{t}})}{t_i}p_i , \\
		B_{ij}({\bf{t}}) &= \frac{B({\bf{t}})}{t_i t_j} p_i p_j -\delta_{ij} \frac{B({\bf{t}})}{t_i t_j}p_i.
	\end{align*}
Thus, the coefficients in \( \mathcal{A}({\bf{t}},w) \) simplify to 
	\begin{align*}
B_{ij}({\bf{t}}) + \frac{F''(B(\bf{t}))}{F'(B(\bf{t}))} 
B_{i}({\bf{t}}) B_{j}({\bf{t}})=\alpha\frac{B({\bf{t}})}{t_i t_j}  p_i p_j-  \delta_{ij}\frac{B({\bf{t}})}{t_i t_j} p_i 
	\end{align*}
Since the covariance matrix satisfies \( \mathrm{cov}(\xi_{i}, \xi_{i}) = I_{k_i} \), we obtain  
	\begin{align*}
		\mathcal{A}({\bf{t}},w) + \mathcal{B}({\bf{t}},w) &= \sum_{i,j} \frac{B({\bf{t}})}{t_i t_j} p_i p_j 
		\left( (\Re w_{i}) \mathrm{cov}(\xi_{i}, \xi_{j}) (\Re w_{j})^{T} 
		- \alpha (\Re z_{i} w_{i}) \mathrm{cov}(\xi_{i}, \xi_{j}) (\Re z_{j} w_{j})^{T} \right) \\
		&\quad + \sum_j \frac{B({\bf{t}})}{t_j^2} p_j \left( |\Re z_j w_j|^2 - |\Re w_j|^2 \right).
	\end{align*}
Adding \( \mathcal{C}({\bf{t}},w) \) to the above expression, condition \eqref{ABC} takes the form 
	\begin{align*}
		\sum_{i,j} & \frac{B({\bf{t}})}{t_i t_j} p_i p_j 
		\left( (\Re w_{i}) \mathrm{cov}(\xi_{i}, \xi_{j}) (\Re w_{j})^{T} 
		- \alpha (\Re z_{i} w_{i}) \mathrm{cov}(\xi_{i}, \xi_{j}) (\Re z_{j} w_{j})^{T} \right) \\
		&\quad + \sum_j \frac{B({\bf{t}})}{t_j^2} p_j \left( (|\Re z_j w_j|^2 - |\Re w_j|^2) + (|\Im w_{j}|^2 - |\Im z_j w_{j}|^2) \right) \geq 0.
	\end{align*}
Since \( (\Re z)^2 - (\Im z)^2 = \Re(z^2) \), we obtain  
	\begin{align*}
		(|\Re z_j w_j|^2 - |\Re w_j|^2) + (|\Im w_{j}|^2 - |\Im z_j w_{j}|^2) 
		&= \Re(z_j^2 w_j w_j^T) - \Re(w_j w_j^T) \\
		&= \Re \left( (z_j^2 - 1) w_j w_j^T \right).
	\end{align*}
Thus, condition \eqref{ABC} reduces to  
	\begin{align*}
	\sum_{i,j} & \frac{B({\bf{t}})}{t_i t_j} p_i p_j 
	\left( (\Re w_{i}) \mathrm{cov}(\xi_{i}, \xi_{j}) (\Re w_{j})^{T} 
	- \alpha (\Re z_{i} w_{i}) \mathrm{cov}(\xi_{i}, \xi_{j}) (\Re z_{j} w_{j})^{T} \right) \\
	&\quad + \sum_j \frac{B({\bf{t}})}{t_j^2} p_j \Re \left( (z_j^2 - 1) w_j w_j^T \right) \geq 0.
\end{align*}
Finally, replacing  \( w_i \) by \( \frac{t_i}{p_i\sqrt{B(t)}} w_i \)  completes the proof.  
\end{proof}

\begin{proof}[Proof of Theorem~\ref{MBEK}]
	Choosing $z_j=is_j$ in Theorem \ref{MCH}, we obtain that inequality \eqref{nfuComImagi} holds if and only if 
	\begin{align}\label{reim1}
		\sum_{i,j}(\Re w_{i})\mathrm{cov}(\xi_{i}, \xi_{j}) (\Re w_{j})^{T} - \alpha s_is_j(\Im w_{i}) \mathrm{cov}(\xi_{i}, \xi_{j}) (\Im w_{j})^{T} \geq \sum_j \frac{(1 + s_j^2)}{p_j}\Re(w_j w_j^T) 
	\end{align}
	for all $w_j\in\C^{k_j}$, $j=1,\ldots,n$. 
    
We have
	\begin{align*}
		w_j w_j^T &= (\Re w_j + i \Im w_j)(\Re w_j + i \Im w_j)^T\\
		&=(\Re w_j)(\Re w_j)^T - (\Im w_j)(\Im w_j)^T + i \left[ (\Re w_j)(\Im w_j)^T + (\Im w_j)(\Re w_j)^T \right].
	\end{align*}
    Thus 
    \begin{align*}
		\sum_j \frac{(1 + s_j^2)}{p_j}\Re(w_j w_j^T)&=	 \sum_j \frac{(1 + s_j^2)}{p_j} \left[ (\Re w_j)(\Re w_j)^T - (\Im w_j)(\Im w_j)^T \right].
        \end{align*}

    Therefore the local condition (\ref{reim1}) holds if and only if 
    
    \begin{align}
    &\sum_{i,j}(\Re w_{i})\mathrm{cov}(\xi_{i}, \xi_{j}) (\Re w_{j})^{T}     - \sum_j \frac{(1 + s_j^2)}{p_j} (\Re w_j)(\Re w_j)^T \geq \label{qf01}\\
    &\sum_{i,j}\alpha s_is_j(\Im w_{i}) \mathrm{cov}(\xi_{i}, \xi_{j}) (\Im w_{j})^{T} - \sum_j \frac{(1 + s_j^2)}{p_j} (\Im w_j)(\Im w_j)^T. \nonumber
    \end{align}
The inequality (\ref{qf01}) compares two quadratic forms acting on different vectors, i.e., $\Re w_{j}$ and $\Im w_{j}$. Therefore the inequality (\ref{qf01}) holds if and only if the left hand side of (\ref{qf01}) is nonneagtive, and the right hand side of (\ref{qf01}) is nonpositive i.e., 
\begin{align*}
   \frac{1}{\alpha} \left\{ \frac{1+s_j^2}{s_j^2p_j}\mathrm{Id}_{k_j} \right\}_{i,j=1}^{n} \geq \mathrm{cov}(\xi)  \geq \mathrm{diag}\left\{ \frac{1+s_{j}^{2}}{p_{j}}\mathrm{Id}_{k_{j}} \right\}_{j=1}^{n}.
\end{align*}
This completes the proof of Theorem~\ref{MBEK}.
\end{proof}

	\begin{proof}[Proof of Theorem~\ref{MBEKS}]
		Choosing $p_1=\ldots=p_n=p$, $s_1=\ldots=s_n=s=\sqrt{p\lambda_{\min}-1}$ and $\alpha=q/p\geq1$ in Theorem \ref{MBEK}, we obtain that \eqref{MBEKSglob} holds if and only if
			\begin{equation}\label{localNOj}
			\frac{1+s^{2}}{s^{2}q} \mathrm{Id}_{k_{1}+\ldots+k_{n}}\geq \mathrm{cov}(\xi) \geq \frac{1+s^{2}}{p} \mathrm{Id}_{k_{1}+\ldots+k_{n}}.  
		\end{equation}
	From the equation $s=\sqrt{p\lambda_{\min}-1}$  we have $\frac{1+s^2}{p}=\lambda_{\min}$ and also, the condition   \eqref{MBEKSloc} gives
		\begin{align*}
			\lambda_{\max} =\frac{p}{q}\cdot\frac{\lambda_{\min}}{p\lambda_{\min}-1}=  \frac{1+s^{2}}{s^{2}q}
		\end{align*}
		meaning that \eqref{MBEKSloc} implies \eqref{localNOj}.
	\end{proof}

\begin{proof}[Proof of Theorem~\ref{rcor}]
The covariance matrix
\[
\mathrm{cov}(\xi) = 
\begin{bmatrix}
	\mathrm{Id}_n & \rho \mathrm{Id}_n \\
	\rho \mathrm{Id}_n & \mathrm{Id}_n
\end{bmatrix}.
\]
has the smallest eigenvalue $\lambda_{\min}=1-|\rho|$ and the largest  $\lambda_{\max}=1+|\rho|$.
\end{proof}

\begin{proof}[Proof of Theorem~\ref{nhsin}]
	Recall, from Mehler formula \eqref{forfur} we have 
	\begin{equation}
		T_{i\sqrt{t_j-1}}f_j(\xi_j)=\frac{\widehat{g_j}}{\widehat{e_{t_j}}}\left( -\eta_j\right)
	\end{equation}
where $\eta_{j} = \frac{\sqrt{t_{j}-1}}{t_{j}} \xi_{j} $, $g_j=f_je_{t_j}$,  $j=1, \ldots, n$,  and $ e_{t}(x) = e^{-|x|^{2}/2t}$. Choosing $s_j=\sqrt{t_j-1}$ in Theorem \ref{MBEK}, we obtain that \eqref{hutRATIOineq} holds for all $g_{j}\in \mathcal{P}(\mathbb{R}^{k_{j}}) e_{t_{j}}$, $j=1, \ldots, n,$ if and only if \eqref{MBEKloc1} holds. 
\end{proof}

\begin{proof}[Proof of Theorem~\ref{pqHS}]
	The proof follows directly from Theorem \ref{MBEKS} and Mehler formula \eqref{forfur}. Note, the Fourier transform of the function $e_t:\R^n\to\R$ given by $ e_{t}(x) := e^{-|x|^{2}/2t}$ is $ \widehat{e_{t}}(\omega) = t^{n/2}e^{-|\omega|^{2}t/2}$. Setting $g_j=f_je_{p\lambda_{\min}}:\R^{k_j}\to\R$, $j=1,\ldots,n$, and $\mu=\frac{\sqrt{p\lambda_{\min}-1}}{p\lambda_{\min}}$ we have 
	\begin{align*}
		\prod_{j=1}^nT_{i\sqrt{p\lambda_{\min}-1}}f_j(\xi_j)&=\prod_{j=1}^n\frac{\widehat{g_j}}{\widehat{e_{p\lambda_{\min}}}}\left(-\mu \xi_j\right)\\
		&=\prod_{j=1}^n (p\lambda_{\min})^{-\frac{k_j}{2}}e^{|\mu \xi_j|^{2}p\lambda_{\min}/2}\widehat{g}_j(-\mu \xi_j)\\
		&=(p\lambda_{\min})^{-\sum \frac{k_j}{2}}e^{\mu^2|\xi|^{2}p\lambda_{\min}/2}\prod_{j=1}^n \widehat{g}_j(-\mu \xi_j)\\
		&=(p\lambda_{\min})^{-\sum \frac{k_j}{2}}e^{\frac{|\xi|^{2}(p \lambda_{\min}-1)}{2p \lambda_{\min}}}\prod_{j=1}^n \widehat{g}_j(-\mu \xi_j).
	\end{align*}
	Moreover, 
	\begin{align*}
		\prod_{j=1}^nf_j(\xi_j)=\prod_{j=1}^n\frac{g_j}{e_{p\lambda_{\min}}}(\xi_j)=\prod_{j=1}^n e^{\frac{|\xi_j|^{2}}{2p\lambda_{\min}}}g_j(\xi_j)=e^{\frac{|\xi|^{2}}{2p \lambda_{\min}}} \prod_{j=1}^{n} g_{j}(\xi_j)
	\end{align*}
Observe that the two calculations above reformulate inequality \eqref{MBEKSglob} as \eqref{nHS}.
\end{proof}

\begin{proof}[Proof of Theorem~\ref{rhoHSI}]
Theorem \ref{pqHS} for $n=2$ and $\xi_1=X$ and  $\xi_2=X_{\rho}$, i.e.,  two $\rho$-correlated standard Gaussians in $\mathbb{R}^{n}$ reduces to the inequality 
\begin{equation}\label{theor2n=2}
\left\|e^{\frac{|\eta|^{2}p \lambda_{\min}}{2}} \widehat{f}(\eta_{1})\widehat{g}(\eta_{2})\right\|_{q} \leq (p\lambda_{\min})^{\frac{k_1+k_2}{2}}\left\|e^{\frac{|\xi|^{2}}{2p \lambda_{\min}}} f(\xi_{1})g(\xi_{2})\right\|_{p}
\end{equation}
holds for all $f,g \in \mathcal{P}(\mathbb{R}^{n})e_{p\lambda_{\min}}$ under the constraint 
\begin{equation}\label{rhoHSIloc1}
\frac{1}{p\lambda_{\min}}+\frac{1}{q\lambda_{\max}}=1,  \ \ \ \ 1\leq p\leq q.
\end{equation}
Notice that $\lambda_{\min}=1-\rho$ and $\lambda_{\max}=1+\rho$. Therefore  $(p\lambda_{\min})^{\frac{k_1+k_2}{2}}=(p(1-\rho))^n $. First we rewrite the right hand side
\begin{align*}
\left\|e^{\frac{|\xi|^{2}}{2p \lambda_{\min}}} f(\xi_{1})g(\xi_{2})\right\|_{p}&=\left\|e^{\frac{|\xi_1|^{2}+|\xi_2|^{2}}{2p (1-\rho)}} f(\xi_{1})g(\xi_{2})\right\|_{p}\\
&=\left(\int_{\R^n}\int_{\R^n}e^{\frac{|x|^2+|\rho x+\sqrt{1-\rho^2}y|^2}{2 (1-\rho)}}|f(x)g(\rho x+\sqrt{1-\rho^2 }y)|^p\,d\gamma(x)\,d\gamma(y)\right)^{1/p}\\
&=\frac{1}{(2\pi)^{n/p}}\left(\int_{\R^n}\int_{\R^n}|f(x)g(\rho x+\sqrt{1-\rho^2 }y)|^pe^{\frac{|x|^2+|\rho x+\sqrt{1-\rho^2}y|^2}{2 (1-\rho)}-\frac{|x|^2+|y|^2}{2}}\,dx\,dy\right)^{1/p}.
\end{align*}
We make the change of variables $x=\sqrt{1-\rho^2}u$ and $y=v-\rho u$ and we see that the exponential density becomes
\begin{align*}
\frac{|x|^2+|\rho x+\sqrt{1-\rho^2}y|^2}{2 (1-\rho)}-\frac{|x|^2+|y|^2}{2}&=\frac{\rho}{2}\left(\frac{1+\rho}{1-\rho}|x|^2+2\frac{\sqrt{1-\rho^2}}{1-\rho}\langle x,y\rangle +|y|^2\right)\\
&=\frac{\rho}{2}|u+v|^2,
\end{align*}
while $\mathrm{d}x\,\mathrm{d}y = (\sqrt{1-\rho^2})^{n} \mathrm{d}u\,\mathrm{d}v$.
Thus the right hand side becomes
\begin{align*}
\left\|e^{\frac{|\xi|^{2}}{2p \lambda_{\min}}} f(\xi_{1})g(\xi_{2})\right\|_{p}&=\left(\frac{\sqrt{1-\rho^{2}}}{2\pi}\right)^{\frac{n}{p}}\left[\int_{\mathbb{R}^{n}}\int_{\mathbb{R}^{n}}\left|f\left(u\sqrt{1-\rho^{2}}\right)g\left(v\sqrt{1-\rho^{2}}\right) \right|^{p}e^{\frac{\rho}{2}|u+v|^{2}}\mathrm{d}u\mathrm{d}v \right]^{1/p}\\
&=\frac{1}{\left(2\pi \sqrt{1-\rho^{2}}\right)^{\frac{n}{p}}}\left[\int_{\mathbb{R}^{n}}\int_{\mathbb{R}^{n}}\left|f\left( x\right)g\left(y\right) \right|^{p}e^{\frac{\rho |x+y|^{2}}{2(1-\rho^{2})}}\mathrm{d}u\mathrm{d}v \right]^{1/p}.
\end{align*}

Next, we set $c=\frac{\sqrt{p(1-\rho)-1}}{p(1-\rho)}$. Applying the same change of variables as before, i.e., $x=\sqrt{1-\rho^2}u$ and $y=v-\rho u$,  we obtain 
\begin{align*}  &\left\|e^{\frac{(|\xi_1|^{2}+|\xi_2|^{2})(p(1-\rho)-1)}{2p(1-\rho)}} \widehat{f}(c\xi_{1})\widehat{g}(c\xi_{2})\right\|_{q}  = \\
&\left(\frac{\sqrt{1-\rho^{2}}}{2\pi}\right)^{\frac{n}{q}}\left[\int_{\mathbb{R}^{n}}\int_{\mathbb{R}^{n}}\left|\widehat{f}\left(cu\sqrt{1-\rho^{2}}\right)\widehat{g}\left(cv\sqrt{1-\rho^{2}}\right) \right|^{q}e^{q(u,v)}\mathrm{d}u\mathrm{d}v \right]^{1/q},
\end{align*}
where $q(u,v)$ is the function
\begin{align*}
q(u,v)&=\frac{|x|^2+|\rho x+\sqrt{1-\rho^2}y|^2}{2 p(1-\rho)}q (p(1-\rho)-1)-\frac{|x|^2+|y|^2}{2}\\
&=\frac{|x|^2+|\rho x+\sqrt{1-\rho^2}y|^2}{2 (1+\rho)}-\frac{|x|^2+|y|^2}{2} \\
&=-\frac{\rho}{2}|u-v|^2.
\end{align*}

Finally, notice that $c\sqrt{1-\rho^{2}}=\frac{1}{\sqrt{pq}}$, hence the left hand side takes the form 
\begin{align*}
&\left\|e^{\frac{|\xi|^{2}(p (1-\rho)-1)}{2p(1-\rho)}} \widehat{f}(x\xi_{1})\widehat{g}(x \xi_{2})\right\|_{q}\\
&=\frac{1}{\left(2\pi c^{2} \sqrt{1-\rho^{2}}\right)^{\frac{n}{q}}}\left[\int_{\mathbb{R}^{n}}\int_{\mathbb{R}^{n}}\left|\widehat{f}(u)\widehat{g}(v) \right|^{q}e^{-\frac{\rho p q}{2}|u-v|^2}\mathrm{d}u\mathrm{d}v \right]^{1/q}
\end{align*}
and this concludes the proof.
\end{proof}

\begin{proof}[Proof of Theorem~\ref{chaoses}]
Let $f_j$ be a $d_j$-homogeneous Gaussian chaos,
\begin{equation*}
f_j(x)=\sum_{|\alpha|=d_j}c_{\alpha}H_{\alpha}(x).
\end{equation*}
Theorem \ref{MBEK} for $p_1=\ldots=p_n=p$, $\alpha=q/p$ and $s_1=\ldots=s_n=s$, reduces to  
\begin{equation*}
\| \prod_{j=1}^n T_{is}f_j(\xi_j)\|_q\leq \| \prod_{j=1}^n f_j(\xi_j)\|_p
\end{equation*}
if and only if
\begin{equation}\label{rHGCloc}
 \frac{1+s^{2}}{s^{2}q} \mathrm{Id}_{N}  \geq \mathrm{cov}(\xi) \geq  \frac{1+s^{2}}{p}\mathrm{Id}_{N} ,
\end{equation}
where $N=k_{1}+\ldots+k_{n}$. We optimize $s$ in \eqref{rHGCloc}. Since the covariance $\mathrm{cov}(\xi)$ must be between  $\lambda_{\min}\mathrm{Id}_{N}$ and $\lambda_{\max}\mathrm{Id}_{N}$, any $s$ ensuring \ref{rHGCloc} must satisfy 
\begin{align*}
\frac{1+s^2}{s^2 q}\geq \lambda_{\max} \ \ \ \ \ \  \text{and} \ \ \ \ \ \ \frac{1+s^2}{p}\leq \lambda_{\min} . 
\end{align*}
Solving both in $s$ we get $s\leq 1/\sqrt{q\lambda_{\max}-1}$ and $s\leq \sqrt{p\lambda_{\min}-1}$ and thus, the best possible $s$ to choose is
\begin{equation*}
s=\min \{ 1/\sqrt{q\lambda_{\max}-1}, \sqrt{p\lambda_{\min}-1} \}.
\end{equation*}
With this choice of $s$, and using the fact that $f_j$ are Gaussian chaoses, we obtain
\begin{equation*}
\| \prod_{j=1}^n T_{is}f_j(\xi_j)\|_q= \| \prod_{j=1}^n (is)^{d_j}f_j(\xi_j)\|_q= s^{\sum d_j}\| \prod_{j=1}^n f_j(\xi_j)\|_q,
\end{equation*}
which completes the proof.
\end{proof}

	\section*{\textbf{Acknowledgements}}
	P.I. was supported in part by NSF CAREER grant DMS-2152401 and a Simons Fellowship. We thank Grigoris Paouris for providing helpful references and comments.

\end{document}